\documentclass[a4paper]{siamart1116}

\usepackage{mathrsfs}

\usepackage{cleveref}
\usepackage[utf8]{inputenc}
\usepackage{enumitem}
\usepackage{algorithm}
\usepackage{algpseudocode}

\newcommand\TheTitle{%
  Solving rank structured Sylvester and Lyapunov equations}
\newcommand\TheShortTitle{\TheTitle}
\newcommand\TheAuthors{Stefano Massei, Davide Palitta, Leonardo Robol}
\newcommand{\trid}{\mathrm{trid}}
\newcommand{\sign}{\mathrm{sign}}
\newcommand{\vect}{\mathrm{vec}}
\DeclareMathOperator{\tril}{tril}

\headers{\TheShortTitle}{\TheAuthors}

\definecolor{lowrankcolor}{rgb}{.75,.75,.75}

\usepackage{geometry}

\title{\TheTitle}

\author{%
  Stefano Massei\thanks{EPF Lausanne, Switzerland,
    \email{stefano.massei@epfl.ch}} \and
  Davide Palitta\thanks{Dipartimento di Matematica, Universit\`a di Bologna, Bologna, Italy,
    \email{davide.palitta3@unibo.it}} \and
  Leonardo Robol\thanks{ISTI-CNR, Pisa, Italy, 
    \email{leonardo.robol@isti.cnr.it}}
}

\usepackage{amsmath,amssymb}
\usepackage{cite}
\usepackage{mathtools}
\usepackage{tikz}
\usepackage{pgfplots}
\usepackage{pgfplotstable}
\usepackage{geometry}
\usepackage{color}
\usepackage{booktabs}
\usepackage{framed}
\pgfplotsset{compat=1.9}

\pgfplotstableset{
	every head row/.style={before row=\toprule,after row=\midrule},
	clear infinite
}

\usepackage{amsopn}
\DeclarePairedDelimiter{\norm}{\lVert}{\rVert}

\numberwithin{theorem}{section}

\newsiamremark{remark}{Remark}
\newsiamremark{example}{Example}

\renewcommand{\leq}{\leqslant}
\renewcommand{\geq}{\geqslant}
\renewcommand{\tilde}{\widetilde}

\ifpdf
\hypersetup{
  pdftitle={\TheTitle},
  pdfauthor={\TheAuthors}
}
\fi

\begin{document}
\maketitle

\begin{abstract}
  We consider the problem of efficiently solving Sylvester and
  Lyapunov equations of medium and large scale, in case of
  rank-structured data, i.e., when the coefficient matrices and the
  right-hand side have low-rank off-diagonal blocks. This comprises
  problems with banded data, recently studied in
  \cite{haber,Palitta2017}, which often arise in the discretization of
  elliptic PDEs.

  We show that, under suitable assumptions, the quasiseparable
  structure is guaranteed to be numerically present in the solution,
  and explicit novel estimates of the numerical rank of the off-diagonal
  blocks are provided.
  
  Efficient solution schemes that rely on the technology of
  hierarchical matrices are described,
 and several numerical experiments 
  confirm the applicability and efficiency of the approaches. We develop
  a MATLAB toolbox that allows easy replication of the experiments
  and a ready-to-use interface for the solvers. The performances of the
  different approaches are compared, and we show that the new methods
  described are efficient on several classes of relevant problems. 
   \bigskip
   
   
   \bigskip 
   
\end{abstract}

\section{Introduction}
\label{sec:introduction}

We consider the problem of solving Sylvester equations of the form
\begin{equation}\label{eq:lyap}
  AX + XB = C,
\end{equation}
where 
 $A\in\mathbb{R}^{n_A\times n_A}$, $B\in\mathbb{R}^{n_B\times n_B}$, $C\in\mathbb{R}^{n_A\times n_B}$
and $A$, $B$ are symmetric positive definite and rank-structured.
More precisely, we assume that the matrices $A$, $B$ and $C$
are \emph{quasiseparable}, i.e., their off-diagonal blocks have low rank. %
 For sake of simplicity, throughout the paper we assume $C$ to be square, 
that is $n_A=n_B\equiv n$, but our results can be easily extended to the case of different $n_A$ and $n_B$.

Sylvester equations arise in different settings, such as problems of
control \cite{benner2008numerical,Antoulas.05}, discretization of PDEs
\cite{Palitta2016, Breiten2014}, block-diagonalization
\cite[chapter 7.1.4]{golub2012matrix},
and many others.  The Lyapunov equation,
that is \eqref{eq:lyap} with $B=A$, is of particular interest due to
its important role in control theory \cite{Antoulas.05}. The symmetric
and positive definite constraint is not strictly necessary in our
analysis, and some relaxations involving the field of values
will be presented. 

Even in the case of sparse $A$, $B$ and $C$, the solution $X$ to
\eqref{eq:lyap} is, in general, dense and it cannot be easily stored for
large-scale problems.  To overcome this numerical difficulty, the
right-hand side is often supposed to be low rank, i.e., $C= C_1 C_2^T$
with $C_1,C_2\in\mathbb{R}^{n\times k}$,
 $k \ll n$.  In this case,
under some suitable assumptions on the spectra of $A$ and $B$, it is
possible to prove that the solution $X$ is numerically low rank
\cite{penzl,antoulas02,grasedyck-existence,beckermann2016} so that it
 can be well-approximated by a low-rank matrix
$X \approx UV^T$. The low-rank property of $X$ justifies 
 the solution of this kind of equations by the so-called low-rank
 methods, which directly compute and store only the factors $U$, $V$.
 A large amount of work in this direction has been carried out in the
 last years. See, e.g., \cite{Simoncini2016} and the references
 therein.  However, in many cases the known term $C$ is not low
 rank. It is very easy to construct a simplified example to show that
 low-rank methods have no hope of being effective in this more general
 context.  Consider equation \eqref{eq:lyap} with
 $A=B = I$ and $C = 2I$ where $I$ denotes the identity matrix.  It is
 immediate to check that the solution is $X = I$, and therefore every
 approximation $UV^T \approx X$ which is not full rank needs to
 satisfy $\lVert UV^T - X \rVert_2 \geq 1$. Obviously, this example
 has no practical relevance from the computational point of view,
 since a Lyapunov equation with diagonal data needs to have a diagonal
 solution, which can be computed in $O(n)$ time and represented in
 $O(n)$ storage.  Nevertheless, it shows that
 even if all the coefficients and the solution $X$ are full rank, they can indeed be very strucutred. One might
 wonder if also banded structures are preserved. This is not true, in
 general, since banded matrices are not an algebra (in contrast to
 what is true for diagonal ones), but approaches which exploit the
 banded properties of $A$, $B$, $C$ and, to a certain extent, of the
 solution $X$, have been recently proposed by Haber and Verhaegen in
 \cite{haber} and by Palitta and Simoncini in \cite{Palitta2017}.
  The preservation of a banded structure in the solution is strictly connected with the conditioning of $A$ and $B$. Unless they are both ill-conditioned, the solution $X$ of \eqref{eq:lyap} is well approximated by a banded matrix $\tilde X$.
 Otherwise,  it has been shown that $X$ can be represented by a couple $(X_B,S_m)$, $X\approx X_B+S_mS_m^T$, where $X_B$ 
 is banded and $S_m$ low-rank so that a low memory allocation is still required, see \cite{Palitta2017}.

 In this work, we consider a more general structure, the so-called
 quasiseparability, which is often numerically present in $X$ when we
 have it in $A$, $B$ and $C$, so that a low memory requirement is
 demanded for storing the solution.  Informally, a matrix is said to
 be quasiseparable if its off-diagonal blocks are low-rank matrices, and the
 quasiseparable rank is defined as the maximum of the ranks of the off-diagonal
 blocks. We say that a matrix is numerically quasiseparable when
 the above property holds only up to a certain $\epsilon$, i.e., only
 few singular values of each off-diagonal block are above a fixed
 threshold.

A simple yet meaningful example arises from the
context of PDEs: consider the differential equation

\begin{equation}\label{Ex1:Laplacian}
    \begin{cases}
      -\Delta u = 
      \log \left( \tau + \left| x - y \right| \right), & 
      (x, y) \in      \Omega,  \\
      u(x,y) \equiv 0, & (x,y) \in \partial \Omega, 
    \end{cases}, \qquad
    \Delta u = \frac{\partial^2 u}{\partial x^2} + \frac{\partial^2 u}{\partial y^2}.
  \end{equation}
  where $\Omega$ is the rectangular domain $[0, 1] \times [0, 1]$ and
  $\tau>0$. The discretization by centered finite
  differences of equation \eqref{Ex1:Laplacian} with $n$ nodes in each
  direction, $(x_i,y_j)$, $i,j=1,\ldots,n,$ yields the following
  Lyapunov equation
 \[
    \begin{array}{c}
    AX + XA = C,\qquad 
    A,C\in\mathbb{R}^{n\times n},\\
    \\
      C_{i,j}= \log \left( \tau + \left| x_i - y_j \right| \right), \\
      \\
      h := \frac{1}{n-1}, \\
    \end{array} \qquad
    A =
    \frac{1}{h^2} \begin{bmatrix}
      2 & -1 \\
      -1 & 2 & -1 \\
      & \ddots & \ddots & \ddots \\
      && -1 & 2 & -1 \\
      &&& -1 & 2 \\
    \end{bmatrix}. 
  \]
  The fact that $A$ is banded implies that it is
  quasiseparable, 
  and also $C$ shares this property.
  Indeed, this follows from the fact that the modulus function it is
  not regular in the whole domain but it is analytic when the sign of
  $x-y$ is constant. This happens in the sub-domains corresponding to
  the off-diagonal blocks.  Separable approximation (and thus
  low-rank) can be obtained by expanding the source $\log(\tau+|x+y|)$
  in the Chebyshev basis. The approximation of this kind of functions
  has been previously investigated in \cite[Chapter 9]{hackbusch2000h}. 
  In Figure~\ref{fig:Clog} (on the right) we
  have reported the decay of the singular values of one off-diagonal
  block of $C$ and $X$ for the case of $\tau = 10^{-4}$ and
  $n=300$. In this case the numerical quasiseparable rank of the
  right-hand side $C$ and the solution $X$ does not exceed $20$ and
  $30$, respectively.  This property holds for any $\tau > 0$: in
  Figure~\ref{fig:Clog} (on the left) we have checked the
  quasiseparable rank of the matrix $C$ for various values of $\tau$,
  and one can see that it is uniformly bounded. The rank is higher
  when $\tau$ is small, because the function is ``less regular'', and
  tends to $1$ as $\tau \to \infty$, because the off-diagonal blocks
  tend to a constant in this case.
  
  \begin{figure}
    \centering
    \begin{tikzpicture}\begin{semilogxaxis}[ylabel = {QS rank},
      xlabel = $\tau$, width = .45\linewidth,
      height = .3\textheight, ymax = 25]
      \addplot table[x index = 0, y index = 1] {tau.dat};
      \legend{QS rank of $C$}
    \end{semilogxaxis}\end{tikzpicture}~\begin{tikzpicture}
      \begin{semilogyaxis}[ylabel = {$\sigma_\ell$} ,
      xlabel = $\ell$, width = .45\linewidth,
        height=.3\textheight]     
        \addplot table[x index = 0, y index = 1] {lap_log.dat};
        \addplot table[x index = 0, y index = 2] {lap_log.dat};
        \addplot [black, dashed, no markers] coordinates {(0,2.2204e-16) (42,2.2204e-16)};
        \legend{$\sigma_l(Y_X)$,
          $\sigma_l(Y_C)$, Unit roundoff};
      \end{semilogyaxis}
    \end{tikzpicture}    
    \caption{On the left, the maximum numerical ranks
      of the off-diagonal blocks
      of the right hand-side $C$ for different values of $\tau$ and $n=300$, using a
      threshold of $10^{-14}$ for truncation. On the right, we set $\tau=10^{-4}$ and the       
      singular values of the off-diagonal blocks
      $Y_C := C(\frac n2+1:n, 1:\frac n2)$ and $Y_X := X(\frac n2+1:n, 1:\frac n2)$ 
    rescaled by the $2$-norm of the two blocks respectively are reported.
     The black dashed line indicates the machine precision
    $2.22\cdot 10^{-16}$.}
    \label{fig:Clog}
  \end{figure}
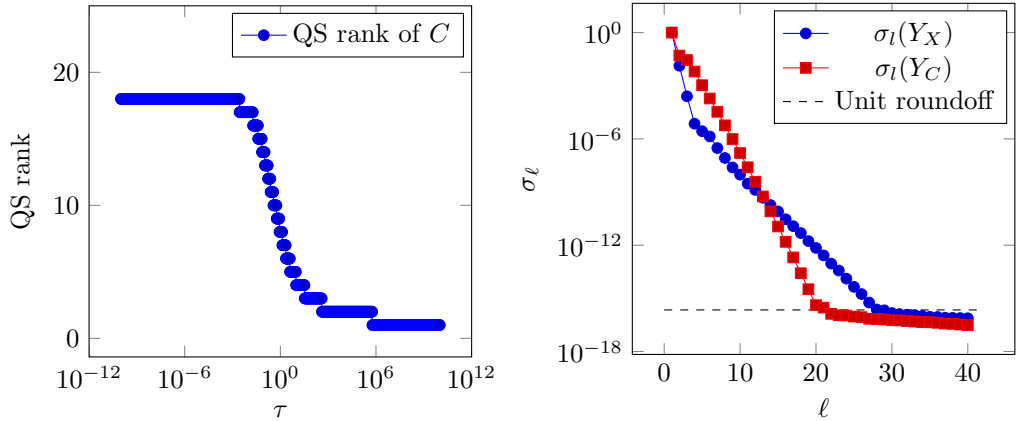

  The problem of solving linear
  matrix equations whose coefficients are represented as
  $\mathcal H$-matrices has already been addressed in \cite{grasedyck-existence,grasedyck-riccati}. In \cite{baur2006factorized}, 
  the authors consider the case of Lyapunov equations with $\mathcal H$-matrices
  coefficients and low-rank right hand side. Recently, in \cite{Bini2017,Bini2016b} the use of hierarchical matrices in the cyclic reduction iteration for solving quadratic matrix equations has been deeply studied.  We will exploit the framework of $\mathcal H$-matrices to store
  quasiseparable matrices and to perform matrix operations at an
  almost linear cost (up to logarithmic factors).

In this paper, we compare the use of hierarchical matrices in the matrix sign iteration, and in the 
estimation of an integral formula for solving \eqref{eq:lyap}. 
The latter approach,  suggested but not  numerically tested in \cite{grasedyck-existence,grasedyck-riccati}, 
relies on evaluating the closed formula  \cite{SaadBoston1990.Birkhauser}

\begin{equation}\label{eq:integral}
X=\int_0^{+\infty}e^{-At}Ce^{-Bt}dt,
\end{equation}
by combining
a numerical integrating scheme and rational approximations for the matrix exponential.  
 We employ \eqref{eq:integral} for our purpose but different closed forms of $X$ 
are available in the literature. See, e.g., \cite{Simoncini2016}.
Starting with $\mathcal H$-matrices representations of $A$, $B$ and $C$, formula~\eqref{eq:integral} can
be efficiently approximated exploiting $\mathcal H$-arithmetic.
To the best of our knowledge, this technique has been exploited only theoretically
for computing $X$ in the $\mathcal H$-matrix framework. On the other hand, exponential
sums are widely used as an approximation tool in the solution of tensor Sylvester
equations \cite{braess2005approximation,braess2009efficient}. 

The representation \eqref{eq:integral} has already been used in
\cite{grasedyck-existence,grasedyck-riccati} as a theoretical tool to estimate 
the quasiseparable rank of the solution, 
but the derived bounds may be very pessimistic, and are
linked with the convergence of the integral formula, which cannot
be easily made explicit.
We improve these estimates by developing a theoretical analysis which
relies on some recent results 
\cite{beckermann2016}, exploited also in
\cite{Bini2017}, where the numerical rank of the
solution $X$ is determined by estimating the exponential decay in the
singular values of its off-diagonal blocks.

The paper is organized as follows; in
Section~\ref{sec:quasisep-solution} we introduce the notion of
quasiseparability and we deliver the technical tools for analyzing the
preservation of the structure in the solution $X$.  In particular, we
provide bounds for the off-diagonal singular values of $X$ and we show
some numerical experiments in order to validate them. In
Section~\ref{sec:hodlr}, hierarchically off diagonal low-rank (HODLR)
matrices are introduced describing their impact in the computational
effort for handling matrix operations.  The two algorithms for solving
\eqref{eq:lyap} are presented in Section \ref{sec:procedure}. In
particular, in Section \ref{Sec_signfun} we recall the sign function
method presented in \cite{grasedyck-riccati}, whereas the procedure used
for the numerical approximation of \eqref{eq:integral} is illustrated
in Section \ref{Sec_integral}.  Both the approaches are based on
the use of 
HODLR arithmetic. We address the solution of certain generalized
Lyapunov and Sylvester equations in Section \ref{Sec_genLyap}.
In Section~\ref{sec:numerical} we 
perform numerical tests on instances of \eqref{eq:lyap} coming from both
artificially crafted models and real-world problems where the
quasiseparable structure is present. Finally, in Section~\ref{sec:conclusion} we draw some concluding remarks.
\section{Quasiseparable structure in the solution}
\label{sec:quasisep-solution}

The main purpose of this section is to prove that, under some 
reasonable assumptions on the spectrum of $A$ and $B$, the solution
$X$ to the matrix equation \eqref{eq:lyap} needs to be quasiseparable
if $A$, $B$ and $C$ are quasiseparable. Throughout the paper we indicate with $\sigma_1(M)\le \sigma_2(M)\le\dots$ the ordered singular values of the matrix $M$.

\subsection{Quasiseparability structures}
\label{sec:quasiseparability-structures}

The literature on quasiseparable (or semiseparable) matrices is
rather large, and the term is often used to denote slightly different
objects. Therefore, also in the spirit of making this paper
as self-contained as possible, we recall the definition of quasiseparable
matrices that we will use throughout the paper. We refer to
\cite{vanbarel:book1,vanbarel:book2,eidelman:book1,vandebril2005bibliography}
and the references therein for a complete survey about quasiseparable
and semiseparable structures. 

\begin{definition}
  A matrix $A$ is \emph{quasiseparable} of order $k$ if the maximum of the ranks of
  all its submatrices contained in the strictly upper or lower part is less or equal than
  $k$. 
\end{definition}
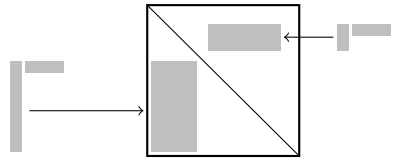
\begin{figure}[!ht]
  \centering
  \begin{tikzpicture}
    \draw [thick] (0,0) rectangle (2,2);
    \draw (0,2) -- (2,0);

    \fill [lowrankcolor] (0.05,1.25) rectangle (0.65,.05);
    \fill [lowrankcolor] (0.8,1.4) rectangle (1.75,1.75);

    \fill [lowrankcolor] (-1.8,1.25) rectangle (-1.65,.05);
    \fill [lowrankcolor] (-1.6,1.25) rectangle (-1.1,1.1);
    \draw[->] (-1.55,0.6) -- (-.05,0.6);

    \fill [lowrankcolor] (2.5,1.75) rectangle (2.65,1.4);
    \fill [lowrankcolor] (2.7,1.75) rectangle (3.2, 1.6);
    \draw [<-] (1.8,1.575) -- (2.45,1.575);
  \end{tikzpicture}
  \caption{Pictorial description of the quasiseparable structure; the
    off-diagonal blocks can be represented as low-rank outer products.}
\end{figure}

\begin{example}
  A banded matrix with bandwidth $k$ is quasiseparable of order (at most)
  $k$. In particular, diagonal matrices are quasiseparable of order
  $0$, tridiagonal matrices are quasiseparable of order $1$, and so on. 
\end{example}

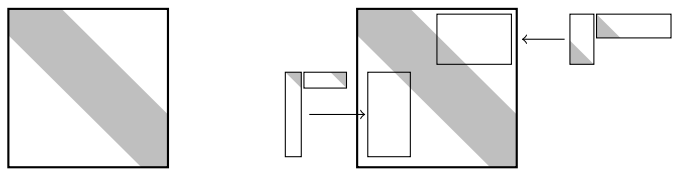
\begin{figure}[!ht]
  \centering
  \begin{tikzpicture}[scale=0.7]
    \fill[lowrankcolor] (0,3) -- (1,3) -- (3,1) -- (3,0) -- (2.5,0) -- (0,2.5) -- (0,3);
    \draw[thick] (0,0) rectangle (3,3);
  \end{tikzpicture}
  \qquad \qquad
  \begin{tikzpicture}[scale=0.7]
    \fill[lowrankcolor] (0,3) -- (1,3) -- (3,1) -- (3,0) -- (2.5,0) -- (0,2.5) -- (0,3);
    \draw[thick] (0,0) rectangle (3,3);

    \draw (0.2,1.8) rectangle (1,0.2);
    \draw (1.5,2.9) rectangle (2.9,1.95);

    \fill[lowrankcolor] (-.5,1.8) -- (-.2,1.8) -- (-.2,1.5) -- cycle;    
    \draw (-1.0,1.8) rectangle (-.2,1.5); 
    \fill[lowrankcolor] (-1.35,1.8) -- (-1.05,1.8) -- (-1.05,1.5) -- cycle;
    \draw (-1.35,1.8) rectangle (-1.05,0.2);
    \draw[->] (-.9,1.0) -- (.15,1);

    \fill [lowrankcolor] (4.0,2.4) -- (4.0,1.95) -- (4.45,1.95) -- cycle;
    \draw (4.0,2.9) rectangle (4.45,1.95);
    \fill [lowrankcolor] (4.5,2.9) -- (4.5,2.45) -- (4.95,2.45) -- cycle;
    \draw (4.5,2.9) rectangle (5.9,2.45);
    \draw [<-] (3.1,2.425) -- (3.9,2.425);
  \end{tikzpicture}
  \caption{Graphic description of the quasiseparability of banded
    matrices; in grey, the non zero entries}
\end{figure}

\subsection{Zolotarev problems and off-diagonal singular values}
\label{sec:zolotarev}

We are interested in exploiting the quasiseparable rank in numerical
computations. In many cases, the request of the exact preservation of a certain
structure is too strong -- and it can not be
guaranteed. However, for computational purposes, we are satisfied if
the property holds in an approximate way, i.e.,
if our data are well-approximated by structured ones. 
This can be rephrased by asking that the off-diagonal blocks of the
solution $X$ of \eqref{eq:lyap} have a low numerical rank. More
precisely, given a generic off-diagonal block of the sought solution
$X$, we want to prove that only a limited number of its singular
values are larger than $\epsilon \cdot \norm{X}_2$, where $\epsilon$ is a
given threshold. This kind of analysis has been already performed in \cite{Bini2017,Bini2016b,Massei2017} for studying the numerical preservation of quasiseparability when solving quadratic matrix equations and computing matrix functions. See also the Ph.D. thesis \cite{massei2017exploiting} for more details.

In order to formalize this approach, we extend a result that provides bounds for the singular values of the solution of 
\eqref{eq:lyap} when the right hand-side has low-rank. 
The latter 
is based on an old problem considered by Zolotarev at the
end of the 19th century \cite{zolotarev1877application},
which concerns
rational approximation in the complex plane.
The following version can be found, along with the proof,
in \cite[Theorem 2.1]{beckermann2016} or
in a similar form in \cite[Theorem 4.2]{Bini2017}.

\begin{theorem}\label{thm:zol1}
  Let $X$ be an $n \times n$ matrix that satisfies the relation
  $AX + XB  = C$, where $C$ is of rank $k$ and $A,B$ are normal matrices.
  Let $E, F$ be two disjoint sets containing the spectra of
  $A$ and $-B$, respectively. Then, the following
  upper bound on the singular values of $X$ holds,
  \[
    \frac{\sigma_{1 + k\ell}(X)}{\sigma_1(X)} \leq Z_\ell(E, F) := \inf_{r(x) \in \mathcal R_{\ell,\ell}}
    \frac{\max_{x \in E} |r(x)|}{\min_{y \in F} |r(y)|}, \qquad
    \ell \geq 1, 
  \]
  where $\mathcal R_{\ell,\ell}$ is the set of rational functions of degree
  at most $(\ell,\ell)$. 
\end{theorem}

Theorem \ref{thm:zol1} provides useful information only if one manages to choose the sets $E$ and $F$ well separated. 
In general it is difficult to explicitly bound $Z_\ell(E, F)$, but
some results exist for specific choices of domains, especially when
$E$ and $F$ are real intervals, see for instance \cite{guttel-polizzi,beckermann2016}.
The combination of these results with Theorem~\ref{thm:zol1} proves the well-known fact that a Sylvester equation with
positive definite coefficients and with a low rank right-hand side has a numerically low-rank solution.

\begin{lemma}\label{lem:sylv-bound}
  Let $A,B$ be symmetric positive definite matrices with
  spectrum contained in $[a,b]$, $0<a<b$. Consider the Sylvester equation
  $AX + XB = C$,
  with $C$ of rank $k$. Then the solution
  $X$ satisfies
  \[
    \frac{\sigma_{1 + k\ell}(X)}{\sigma_1(X)} \leq 4\rho^{-2\ell}
    \]
    where $\rho=\operatorname{exp}\left(\frac{\pi^2}{2\mu(\frac ba)}\right)$ and $\mu(\cdot)$ is the Gr\"otzsch ring function
      \[
      \mu(\lambda):= \frac{\pi}{2} \frac{K(\sqrt{1-\lambda^2})}{K(\lambda)},\qquad K(\lambda):=\int_0^1\frac{1}{(1-t^2)(1-\lambda^2t^2)}dt.
      \]
\end{lemma}
\begin{proof}
Applying Theorem~\ref{thm:zol1} with $E=[a,b]$ and $F=[-b,-a]$ we get
\[
\frac{\sigma_{1 + k\ell}(X)}{\sigma_1(X)} \leq Z_\ell(E,F).
\]
Using Corollary 3.2 in \cite{beckermann2016} for bounding $Z_\ell(E,F)$ we get the claim.
\end{proof}
\begin{remark}
A slightly weaker bound which does not involve elliptic functions is the following~\cite{beckermann2016}
\[
Z_\ell([a,b],[-b,-a])\leq 4\rho^{-2\ell},\qquad 
\rho=\operatorname{exp}\left(\frac{\pi^2}{2\log\left(4\frac ba\right)}\right),\qquad 0<a<b<\infty.
\]  
\end{remark}

It is easy to see that in case of Lyapunov equations with symmetric
positive definite coefficients we can replace the quantity $\frac ba$
with the condition number of $A$.
\begin{corollary}\label{lem:lyap-bound}
  Let $A$ be a symmetric positive definite matrix with
  condition number $\kappa_A$, and consider the Lyapunov equation
  $AX + XA = C$, with $C$ of rank $k$. Then the solution
  $X$ satisfies
  \[
    \frac{\sigma_{1 + k\ell}(X)}{\sigma_1(X)} \leq 4\rho^{-2\ell}
  \]
  where $\rho=\operatorname{exp}\left(\frac{\pi^2}{2\mu(\kappa_A)}\right)$ and $\mu(\cdot)$ is defined as 
  in Lemma~\ref{lem:sylv-bound}.
\end{corollary}

We are now interested in proving that the solution of a Sylvester
equation with low-order quasiseparable data is numerically
quasiseparable.  An analogous task had been addressed in
\cite{grasedyck-existence}.  The approach developed by the authors can
be used for estimating either the rank of $X$ in the case of a
low-rank right-hand side or the rank of the off-diagonal blocks of $X$
when the coefficients are hierarchical matrices.  In particular, it
has been shown that if the coefficients are efficiently represented by
means of the hierarchical format then also the solution shares this
property. The estimates provided in \cite{grasedyck-existence} exploit
the convergence of a numerical integrating scheme for evaluating the
closed integral formula \eqref{eq:integral}. These bounds are however
quite implicit, and 
are more pessimistic than the estimates provided in
\cite{penzl}, and in \cite{sorensen2002bounds} for the case of
a low-rank right hand-side (which is the setting where
all the previous results are applicable). 

Here, we directly characterize the off-diagonal singular values of the solution applying Theorem~\ref{thm:zol1} block-wise. 
\begin{theorem}
  \label{thm:quasisep-sol-normal}
  Let $A$ and $B$ be  symmetric positive definite matrices of quasiseparable
  rank $k_A$ and $k_B$, respectively, and
  suppose that the spectra of $A$ and $B$ are both contained in the interval $[a,b]$. Then,
  if $X$ solves the Sylvester equation $AX + XB = C$, with $C$ of
  quasiseparable rank $k_C$, a generic off-diagonal block $Y$ of
  $X$ satisfies
  \[
    \frac{\sigma_{1 + k\ell}(Y)}{\sigma_1(Y)} \leq 4\rho^{-2\ell},
  \]
  where $k:=k_A+k_B+k_C$,
 $\rho=\operatorname{exp}\left(\frac{\pi^2}{2\mu(\frac ba)}\right)$ and $\mu(\cdot)$ is defined as in Lemma~\ref{lem:sylv-bound}.
\end{theorem}
\begin{proof}
Consider the following block partitioning for the Lyapunov equation
\[
  \begin{bmatrix}
    A_{11} & A_{12} \\
    A_{21} & A_{22} \\
  \end{bmatrix}
  \begin{bmatrix}
    X_{11} & X_{12} \\
    X_{21} & X_{22} \\
  \end{bmatrix} +
\begin{bmatrix}
    X_{11} & X_{12} \\
    X_{21} & X_{22} \\
  \end{bmatrix}
  \begin{bmatrix}
    B_{11} & B_{12} \\
    B_{21} & B_{22} \\
  \end{bmatrix} =
  \begin{bmatrix}
    C_{11} & C_{12} \\
    C_{21} & C_{22} \\    
  \end{bmatrix}. 
\]
where the off-diagonal blocks --- in each matrix --- do not involve any elements of the main diagonal and all the dimensions 
are compatible.
Without loss of generality we can consider the case $Y=X_{21}$. Observe that, writing the above system block-wise we get the
following relation
\[
A_{21} X_{11} + A_{22} X_{21} + X_{21} B_{11} + X_{22} B_{21} = C_{21}.
\]
In particular the block $X_{21}$ solves the Sylvester equation
\[
A_{22} X_{21} + X_{21} B_{11}  = C_{21}-A_{21} X_{11} - X_{22} B_{21},
\]
in which the right-hand side has (standard) rank bounded by $k$. Since $A_{22}$ and $B_{11}$
are principal submatrices of symmetric positive definite matrices, they are again symmetric positive definite and such that 
$\kappa_2(A_{22}) \leq \frac ba$, and
$\kappa_2(B_{11}) \leq \frac ba$.  Therefore, using
Lemma~\ref{lem:sylv-bound} we get the claim.
\end{proof}

\begin{remark} \label{rem:bandedrank}
	In the case where $A,B$ and $C$ are banded with bandwidth $k_A,k_B$ and $k_C$, respectively, one can refine the bound given in Theorem~\ref{thm:quasisep-sol-normal} by using 
	$k := \max\{ k_A + k_B, k_C \}$. Indeed, being $A_{21}$ the off-diagonal block of a banded matrix, it has a rows generator with non zero entries only in the first $k_A$ rows. Analogously, the non zero entries of the columns generator corresponding to $B_{21}$, are located in its last $k_B$ rows. Finally, non zero entries of $C_{21}$ are in the $k_C$ diagonals located in the upper right corner. 
	Therefore, the matrix
	$
	C_{21}-A_{21}X_{11}-X_{22}B_{21}
	$
	has non zero elements only on the first $k_A$ rows, in the last $k_B$ columns and in the $k_C$ upper right corner diagonals, see Figure~\ref{fig:diagonals}. This provides the upper bound $\max\{k_A+k_B,k_C\}$ for its rank.
	\end{remark}
	
In Figure~\ref{fig:bound} we compare the bound given in Theorem~\ref{thm:quasisep-sol-normal} with the off-diagonal singular
values of the solution.
In this experiment the matrix $C\in\mathbb R^{n\times n}$, $n=300$, is diagonal with random entries and $A=B=MM^T$ where
$M\in\mathbb R^{n\times n}$ is bidiagonal with ones on the main diagonal and random elements -- chosen in $(0,1)$ -- 
in the subdiagonal. The theoretical bound manages to describe the superlinear decay of the off-diagonal singular values. 
On the other hand, there is a significant gap between this estimate and the real behavior of the singular values.
This is due to the fact that we are bounding the quantity $Z_\ell(E,F)$ where $E$ and $F$ are the convex hull of 
the spectra of
$A$ and $-B$ respectively, instead of considering the Zolotarev problem directly on the discrete spectra. This is done in order to 
find explicit bounds but it can cause an overestimation as outlined in \cite{beck-gryson}. 

\begin{figure} 
	\centering
	\begin{tikzpicture}[scale=0.8]
		\draw (0,0) rectangle (2.0,2.0); \node at (1,1) {$C_{21}$}; 
		\fill[gray] (1.2,2) -- (2,1.2) -- (2,2);
		
		\node at (2.5,1) {$-$};
		
		\draw (3.0,0) rectangle (5,2); \node at (4,1) {$A_{21} X_{11}$};
		\fill[gray] (3,1.7) rectangle (5,2);
		
		\node at (5.5,1) {$-$};
		
		\draw (6,0) rectangle (8,2); \node at (7,1) {$X_{22} B_{21}$};
		\fill[gray] (7.8,0) rectangle (8,2);
		
		\node at (9,1) {$=$};
		
		\draw (10,0) rectangle (12,2); 
		\fill[gray] (10,1.7) rectangle (12,2);
		\fill[gray] (11.8,0) rectangle (12,2);
		\fill[gray] (11.2,2) -- (12,1.2) -- (12,2);
	\end{tikzpicture}
	\caption{Sparsity structure of the equation for the off-diagonal block
		$X_{21}$ when $A,B$, and $C$ are banded matrices. As described in 
		Remark \ref{rem:bandedrank} the rank of the right handside is bounded by
		$\max\{k_A+k_B,k_C\}$.}
	\label{fig:diagonals}
\end{figure}
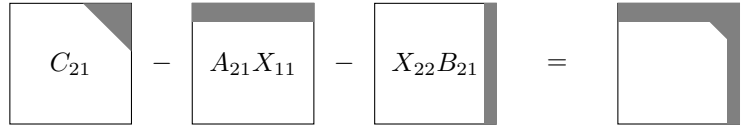

\begin{figure}\label{fig:bound}
    \centering
    \begin{tikzpicture}
      \begin{semilogyaxis}[ylabel = $\sigma_j$ , xlabel = $j$, width = .8\linewidth, height=.3\textheight]     
        \addplot [skip coords between index={13}{150}, mark = *, color = blue] table[x index = 0, y index = 1] {random.dat};
        \addplot table[x index = 0, y index = 2] {random.dat};
        \legend{Off-diagonal singular values of $X$, Bound from Theorem~\ref{thm:quasisep-sol-normal}};
      \end{semilogyaxis}
    \end{tikzpicture} 
    \caption{Off-diagonal singular values in the solution $X$ to \eqref{eq:lyap} where $C$ is a random diagonal 
    matrix and $A=B=MM^T$ with $M$ bidiagonal matrix with ones on the main diagonal and random elements -- chosen in $(0,1)$ --
    in the subdiagonal. 
    The dimension of the matrices is $n\times n$ with $n=300$. 
    The blue dots represent the most significant singular values of the off-diagonal block $X(\frac n2+1:n,1:\frac n2)$.
    The red squares represent the theoretical bound given by Theorem~\ref{thm:quasisep-sol-nonnormal}.}
    \end{figure}
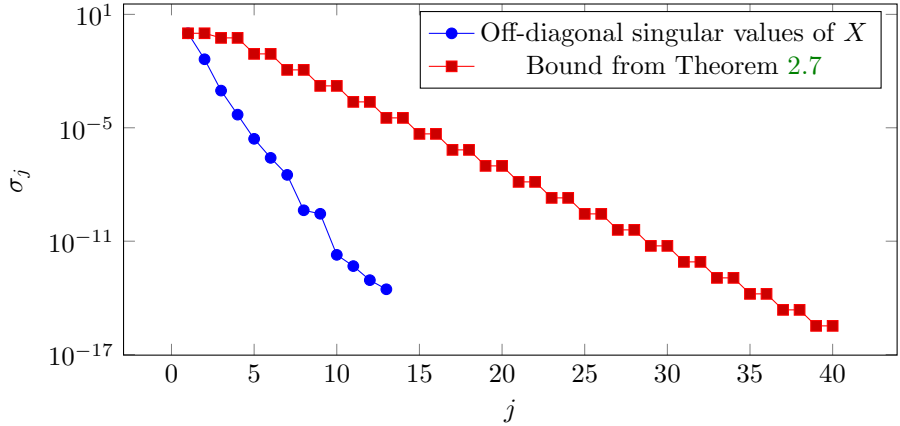
    
A key property in the proof of Theorem~\ref{thm:quasisep-sol-normal} is the fact that submatrices
of positive definite matrices
are better conditioned than the original ones. This
is an instance of a more general situation, which we can use
to characterize the solution of Sylvester equations with non-normal
coefficients.

\begin{definition}
  Given an $n \times n$ square matrix $A$ we say that its \emph{field
    of values} is the subset of the complex plane defined as
  follows:
  \[
    \mathcal W(A) := \left\{ \frac{x^H A x}{x^H x} \ \Big| \ x \in \mathbb {C}^n \backslash \{ 0 \} \right\}. 
  \]
\end{definition}

One can easily check that for a normal matrix, being unitarily diagonalizable,
the field of values is just the convex hull of the eigenvalues. For a general
matrix, we know that the spectrum is contained in $\mathcal W(A)$, but
the latter can be strictly larger than the convex hull of the former.

\begin{lemma}\label{lem:proj}
  Let $P$ be an orthogonal projection, i.e., a $n \times k$ matrix, $k < n$, with orthonormal columns. 
  Then, for any matrix $A$,
  $\mathcal W(P^H A P) \subseteq \mathcal W(A)$. In particular, the
  field of values of any principal submatrix of $A$ is contained in $\mathcal W(A)$. 
\end{lemma}

\begin{proof}
  The result directly comes by observing that
  \[
    \max_{y \in \mathbb{C}^{k}} \frac{y^H P^H A P y}{y^H y} =
    \max_{y \in \mathbb{C}^{k}} \frac{y^H P^H A P y}{y^H P^H P y} 
    \overbrace{\leq}^{x = Py} \max_{x \in \mathbb{C}^{n}} \frac{x^H A x}{x^H x}.
  \]
\end{proof}

\begin{lemma}[Crouzeix \cite{last-crouzeix}]\label{lem:crouzeix}
  Let $A$ be any $n \times n$ matrix, and $f(z)$ an holomorphic
  function defined on $\mathcal W(A)$. Then,
  \[
    \norm{f(A)}_2 \leq \mathcal C \max_{z \in \mathcal W(A)} |f(z)|, 
  \]
  where $\mathcal C$ is a universal constant smaller or equal than $1+\sqrt 2$. 
\end{lemma}

The above result is conjectured to be true with $\mathcal C = 2$, and in this
form is often referred to as the \emph{Crouzeix conjecture}
\cite{crouzeix2007numerical}. Lemma~\ref{lem:proj}-\ref{lem:crouzeix} can be exploited
to obtain a generalization of Theorem~\ref{thm:quasisep-sol-normal}.

\begin{theorem}
  \label{thm:quasisep-sol-nonnormal}
  Let $A, B$ be matrices of quasiseparable
  rank $k_A$ and $k_B$ respectively and such that $\mathcal W(A) \subseteq E$ and 
  $\mathcal W(-B) \subseteq F$. Consider 
  the Sylvester equation $AX + XB = C$, with $C$ of
  quasiseparable rank $k_C$. Then a generic off-diagonal block $Y$ of
  the solution $X$ satisfies
  \[
    \frac{\sigma_{1 + k\ell}(Y)}{\sigma_1(Y)} \leq \mathcal C^2\cdot Z_{\ell}(E, F), \qquad k:=k_A+k_B+k_C. 
  \]
\end{theorem}
Other similar extensions of this result can be obtained using the theory of $K$-spectral sets \cite{badea2013spectral}. 
\subsection{Quasiseparable approximability}
In the previous section we showed that, when the coefficients of the
Sylvester equation are quasiseparable, the off-diagonal blocks of
the solution $X$ have quickly decaying singular values. We want to
show that this property implies the existence of a quasiseparable approximant.

In order to do that, we first introduce the definition of
$\epsilon$-quasiseparable matrix. 

\begin{definition}
	We say that $A$ has \emph{$\epsilon$-quasiseparable rank $k$} if, for
	every off-diagonal block $Y$, $\sigma_{k+1}(Y) \leq \epsilon$. If the
	property holds for the lower (resp. upper) offdiagonal blocks,
	we say that $A$ has lower (resp. upper) $\epsilon$-quasiseparable rank $k$. 
\end{definition}

\begin{remark} \label{rem:epsilon-subblock}
	Notice that, if a matrix $A$ has $\epsilon$-quasiseparable rank $k$,
	then the same property is true for any of its principal submatrix $A'$.
	In fact, any off-diagonal block $Y$ of $A'$ is also an off-diagonal
	block of $A$, and therefore $\sigma_{k+1}(Y) \leq \epsilon$. 
\end{remark}

The next step is showing that an $\epsilon$-quasiseparable matrix admits a quasiseparable approximant. 
First, we need the following technical Lemma where $\oplus$ denotes the direct sum. 

\begin{lemma} \label{lem:unitary-quasisep}
	Let $A$ be a matrix with $\epsilon$-quasiseparable rank $k$,
	$Q$ any $(k+1) \times (k+1)$
	unitary matrix. Then, $(I_{n-k-1} \oplus Q) A$ also has
	$\epsilon$-quasiseparable rank $k$. 
\end{lemma}

\begin{proof}
	We prove the result for the lower off-diagonal blocks; the proof for
	the upper part follows the same lines. Observe that we can verify
	the property for every maximal subdiagonal block $Y$, that is $Y$
	involves the first sub-diagonal and the lower-left corner. If $Y$ is contained
	in the last $k$ rows, its rank is at most $k$.  Otherwise, if $Y$
	includes elements from the last $j > k$ rows, we can write
	$Y = (I_{j-k-1} \oplus Q) \tilde Y$, where $\tilde Y$ is the
	corresponding subblock of $A$ (these two situations
	are depicted in Figure~\ref{fig:offdiagonal-after-transform}). Therefore,
	$\sigma_{k+1}(Y) = \sigma_{k+1}(\tilde Y) \leq \epsilon$.
\end{proof}

\begin{figure}
	\centering
	\begin{tikzpicture}
	\draw (0,0) rectangle (2,2);
	\draw (0,2) -- (1.6,0.4);
	\draw (1.6,0) rectangle (2,0.4);
	\node at (1.8,0.2) { $Q$ };
	
	\draw (2.2,0) rectangle (4.2,2);
	\draw (2.2,2) -- (3.1,1.1); \draw (3.4,0.8) -- (4.2,0);    
	\node at (3.2,1) { $A$ };
	
	\draw[dashed,gray] (2.2,1.3) -- (2.9,1.3) -- (2.9,0);
	\draw[dashed,gray] (2.2,0.3) -- (3.9,0.3) -- (3.9,0);    
	
	\node at (4.6,1) { $=$ };
	
	\draw (5,0) rectangle (7,2);
	\draw (5,2) -- (7,0);
	
	\draw[dashed,gray] (5,1.3) -- (5.7,1.3) -- (5.7,0);
	\draw[dashed,gray] (5,0.3) -- (6.7,0.3) -- (6.7,0);
	\end{tikzpicture}
	\caption{Off-diagonal blocks in the matrix $(I_{n-k-1} \oplus Q) A$. From the
		picture one sees that the $Q$ acts on the tall block without
		changing its singular values, and that the small one has small
		rank thanks to the small number of rows.}
	\label{fig:offdiagonal-after-transform}
\end{figure}
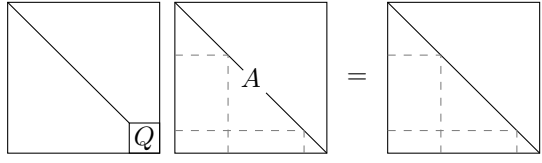

\begin{theorem} \label{thm:quasisep-approximant}
	Let $A$ be of $\epsilon$-quasiseparable rank $k$, for $\epsilon > 0$.
	Then, there exists a matrix $\delta A$ of norm
	bounded by $\norm{\delta A}_2 \leq 2 \sqrt{n} \cdot \epsilon$
	so that $A + \delta A$ is $k$-quasiseparable. 
\end{theorem}

\begin{proof}
	We first show that there exists a 
	perturbation $\delta A_\ell$ of norm bounded by $\sqrt{n} \cdot \epsilon$
	that makes every lower off-diagonal block of $A$ of rank $k$.
	
	We prove the result by induction on the dimension of $A$.
	If $n \leq 2k+1$ there is nothing to prove, since $A$ has
	all the off-diagonal blocks of rank at most $k$. If $n \geq 2k+2$,
	consider the following block partitioning of $A$:
	\[
	A =
	\begin{bmatrix}
	A_{11}  & A_{12} \\
	A_{21} & A_{22} \\
	\end{bmatrix}, \qquad
	A_{11} \in \mathbb{C}^{(n-k-1) \times (n-k-1)}, \quad 
	A_{22} \in \mathbb{C}^{(k+1) \times (k+1)}. 
	\]
	Since $\sigma_{k+1}(A_{21}) \leq \epsilon$, multiplying on the left by
	a unitary matrix $I_{n-k-1} \oplus Q^T$, where $Q$ contains the first
	$k$ left singular vectors of $A_{21}$, yields
	\[
	\widetilde A := ( I_{n-k-1} \oplus Q^T ) A = 
	\begin{bmatrix}
	\widetilde A_{1} & v \\
	w^T & d \\
	\end{bmatrix}, \qquad
	w =
	\begin{bmatrix}
	w_1 \\
	w_2 
	\end{bmatrix}, \quad 
	\norm{w_1}_2 \leq \epsilon, \quad 
	w_2 \in \mathbb{C}^{k}, \quad
	d \in \mathbb{C}. 
	\]
	Observe that, in view of Lemma~\ref{lem:unitary-quasisep},
	$\widetilde A$ still has $\epsilon$-quasiseparable rank $k$ and,
	according to Remark~\ref{rem:epsilon-subblock}, the same holds for
	$\widetilde A_1$. Therefore, thanks to the induction step, there
	exists $\delta \widetilde A_{\ell,1}$ such that
	$\widetilde A_1 + \delta \widetilde A_{\ell,1}$ has lower
	quasiseparable rank $k$ and
	$\norm{\delta \widetilde A_{\ell,1}}_2 \leq \sqrt{n-1} \cdot \epsilon$.
	
	Define $\delta  A_\ell$ and $\delta \widetilde A_\ell$ as follows:
	\[
	\delta  A_\ell :=
	( I_{n-k-1} \oplus Q ) \underbrace{\begin{bmatrix}
		\delta \widetilde A_{\ell,1} & 0 \\
		- z^T & 0 \\
		\end{bmatrix}}_{\delta \widetilde A_\ell}, \qquad 
	z =
	\begin{bmatrix}
	w_1 \\
	0 \\
	\end{bmatrix}. 
	\]
	Notice that  
	$\norm{\delta \widetilde A_\ell}_2 \leq \sqrt{ \norm{\delta \widetilde A_{\ell,1}}_2^2 + \norm{z}_2^2
	} \leq \sqrt{n} \epsilon$. We claim that $A + \delta A_\ell$ is
	lower $k$-quasiseparable. With a direct computation we get
	\[
	A + \delta A_\ell = ( I_{n-k-1} \oplus Q )
	\underbrace{\begin{bmatrix}
		\tilde A_1 + \delta \tilde A_{\ell,1} & v \\
		w^T - z^T & d \\
		\end{bmatrix}}_{\widehat A}. 
	\]
	The matrix $\widehat A$ is lower $k$-quasiseparable. In fact, every
	subdiagonal block of $\widehat A$ is equal to a subblock of
	$\widetilde A_1 + \delta \widetilde A_{\ell,1}$, possibly with an additional last row. If
	the subblock does not involve the last $k + 1$ columns, the
	additional row is zero, and so the rank does not increase.
	Otherwise, the smallest dimension of the block is less or equal than
	$k$, so its rank is at most $k$. Applying
	Lemma~\ref{lem:unitary-quasisep} once more with $\epsilon=0$ we get
	that $A+\delta A_\ell$ is lower $k$-quasiseparable.
	Notice that, it is not restrictive to assume $\delta A_{\ell}$ lower triangular. In fact, if this is not the case one can consider $\tril(\delta A_\ell)$ that still has the same property and has a smaller norm. 
	
	Repeating the process with $A^T$, we obtain an upper triangular matrix $\delta A_u$,
	of norm bounded by $\sqrt{n} \cdot \epsilon$, such that
	$A + \delta A_u$ is upper $k$-quasiseparable. Therefore, we have
	that $A + \delta A$ with
	$\delta A := \delta A_\ell + \delta A_u$ is $k$-quasiseparable, and 
	$\norm{\delta A}_2 \leq \norm{\delta A_\ell}_2 +
	\norm{\delta A_u}_2 \leq 2 \sqrt{n} \cdot \epsilon$. 
\end{proof}
\begin{remark}
	Notice that,  for $n\leq 2k+1$, the claim of  Theorem~\ref{thm:quasisep-approximant} holds by choosing $\delta A=0$.
	This means that the constant $2\sqrt{n}$ can be replaced with $2\sqrt{\max\{n-2k-1,0\}}$.
\end{remark}
The above result shows that a matrix with $\epsilon$-quasiseparable
rank of $k$ can be well-approximated by a matrix with exact
quasiseparable rank $k$.
\subsection{Preservation of the quasiseparable and banded structures}
The results of the previous section guarantee the presence of a numerical quasiseparable structure in the solution $X$ to
\eqref{eq:lyap} when the spectra of $A$ and $-B$ are
well separated in the sense of the Zolotarev problem. 

The preservation of a banded pattern in the solution has already been treated in \cite{haber,Palitta2017} in case of Lyapunov
equations with banded data and well-conditioned coefficient matrix. Moreover, in \cite{Palitta2017}, it has been shown that 
if $A$ is ill-conditioned, the solution $X$ can be written as the sum of a banded matrix and a low-rank one, so that $X$ 
is quasiseparable.
It is worth noticing that the results concerning the preservation of the banded and the banded plus low-rank structures
do not require the separation property 
on the spectra of the coefficient matrices. This means that there are cases ---not covered by the results of
Section~\ref{sec:zolotarev}--- where the quasiseparability is still preserved. 
 
In order to validate this consideration we set up some experiments
concerning the solution to \eqref{eq:lyap} varying the structure of
the coefficients and of the right-hand side.  In particular, the
features of the solution we are interested in are: the distribution of
the singular values $\sigma_\ell$ of the off-diagonal block
$X(\frac n2+1:n,1:\frac n2)$\footnote{Notice that, in order to obtain a
  good hierarchical representation of the given matrices, the same structure
  needs to be present also in the upper off-diagonal block, and in the
  smaller off-diagonal blocks obtained in the recursion. Here we check
  just the larger off-diagonal block for simplicity; in the generic case, one may expect the quasiseparable rank 
  to be given by the rank of this block. 
  }
  and the decay in the magnitude of the
elements getting far from the main diagonal. The latter quantity is
represented with the distribution of the maximum magnitude along the
subdiagonal $\ell$ as $\ell$ varies from $1$ to $n$.  In all the
performed tests we set $n=300$ and the solution $X$ is computed by the
Bartels-Stewart algorithm \cite{Bartels1972}.

\begin{figure}\label{fig:test1}
    \centering
    \begin{tikzpicture}
      \begin{semilogyaxis}[ylabel = $\sigma_\ell$ , xlabel = $\ell$, xmax=50, width = .95\linewidth,
        height=.35\textheight, legend pos=south west]     
        \addplot table[x index = 0, y index = 1] {test2-3.dat};
        \addplot table[x index = 0, y index = 7] {test2-3.dat};
         \addplot table[x index = 0, y index = 2] {test2-3.dat};
                \addplot table[x index = 0, y index = 8] {test2-3.dat};
        \legend{$\sigma_\ell$ (banded RHS), decay in the band (banded RHS), $\sigma_\ell$ (dense quasisep. RHS), 
        decay in the band (dense quasisep. RHS)};
      \end{semilogyaxis}
    \end{tikzpicture}
    \caption{We compute $X_i\in\mathbb{R}^{n\times n}$, $n=300$, as the solution of $AX_i+X_iA= C_i$ for $i=1,2$ respectively. 
    $A$ is symmetric and tridiagonal with eigenvalues in $[0.2,+\infty)$ 
    (positive definite and well-conditioned). $C_1$ is tridiagonal symmetric while $C_2$ is 
    a dense random symmetric quasiseparable matrix of rank $1$.}
    \end{figure}
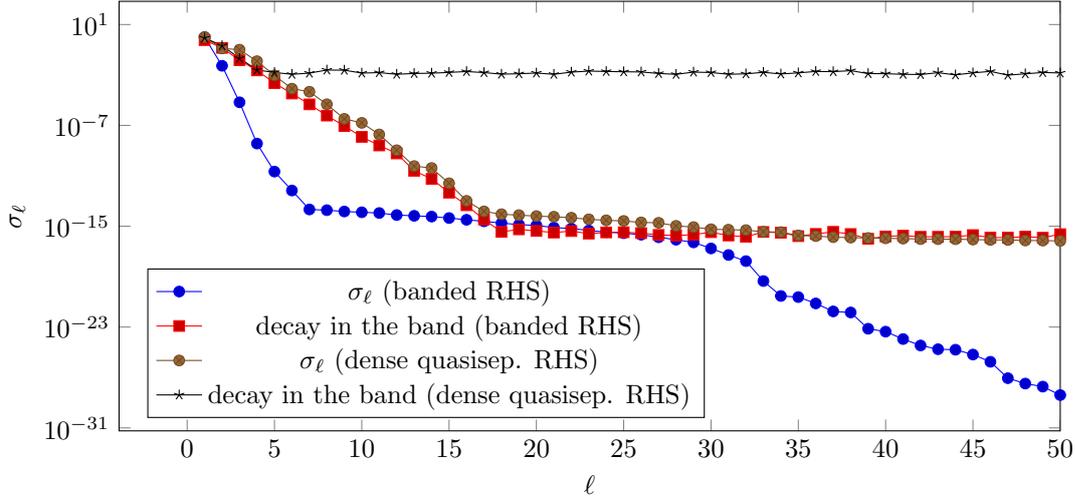
    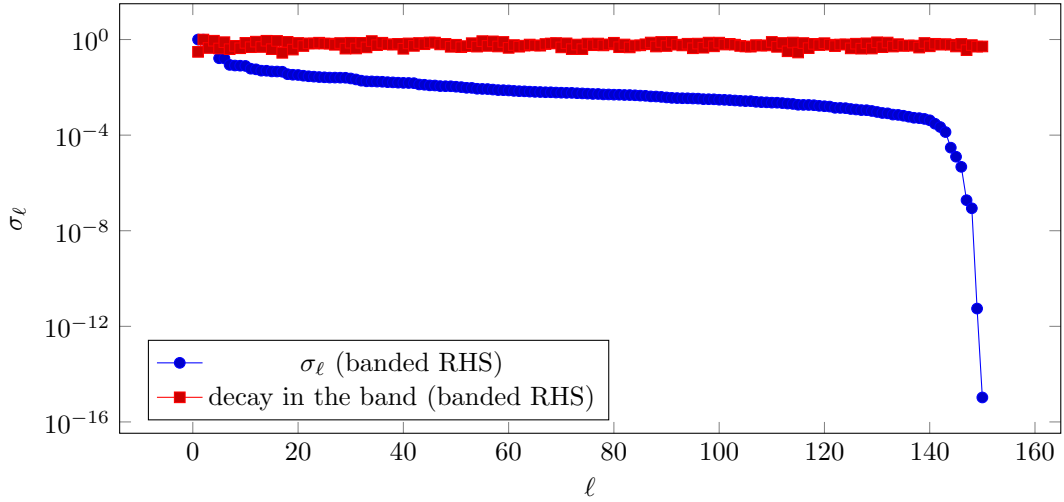
\begin{figure}\label{fig:test2}
        \centering
        \begin{tikzpicture}
          \begin{semilogyaxis}[ylabel = $\sigma_\ell$ , xlabel = $\ell$, width = .95\linewidth,
            height=.35\textheight, legend pos=south west]     
            \addplot table[x index = 0, y index = 3] {test2-3.dat};
            \addplot table[x index = 0, y index = 9] {test2-3.dat};
            \legend{$\sigma_\ell$ (banded RHS), decay in the band (banded RHS)};
          \end{semilogyaxis}
        \end{tikzpicture}
        \caption{We compute the solution $X$ of $AX+XA = C$ and we analyze the off-diagonal block $X(\frac n2+1:n,1:\frac n2)$. $A=\trid(-1,2,1)-1.99\cdot I$
        (indefinite and ill-conditioned) while $C$ is a random diagonal matrix.}
        \end{figure}
        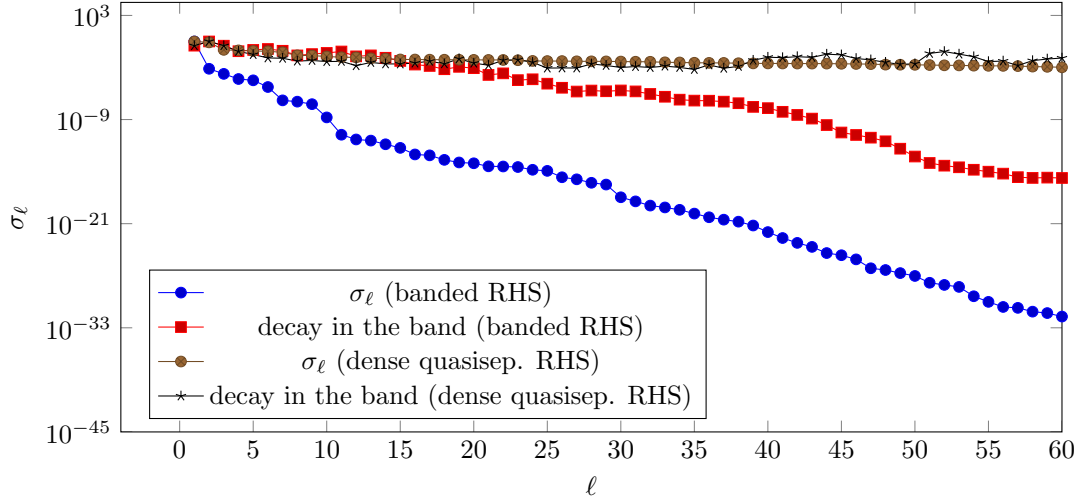
\begin{figure}\label{fig:test3}
            \centering
            \begin{tikzpicture}
              \begin{semilogyaxis}[ylabel = $\sigma_\ell$ , xlabel = $\ell$,xmax=60, xmin=-4, ymin=1e-45,width = .95\linewidth,
                height=.35\textheight, legend pos=south west]     
                \addplot table[x index = 0, y index = 4] {test2-3.dat};
                \addplot table[x index = 0, y index = 10] {test2-3.dat};
                 \addplot table[x index = 0, y index = 5] {test2-3.dat};
                        \addplot table[x index = 0, y index = 11] {test2-3.dat};
                \legend{$\sigma_\ell$ (banded RHS), decay in the band (banded RHS), $\sigma_\ell$ (dense quasisep. RHS),
                decay in the band (dense quasisep. RHS)};
              \end{semilogyaxis}
            \end{tikzpicture}
            \caption{We compute $X_i\in\mathbb{R}^{n\times n}$, $n=300$, as the solution of $AX_i-X_iB= C_i$ for $i=1,2$,
            respectively. 
             $A$ and $B$  are symmetric and tridiagonal with eigenvalues in $[0.2,14]$ and $[0.5,14]$ 
            (well conditioned but without separation of the spectra). $C_1$ is tridiagonal symmetric while
            $C_2$ is a dense random symmetric quasiseparable matrix of rank $1$.}
            \end{figure}
            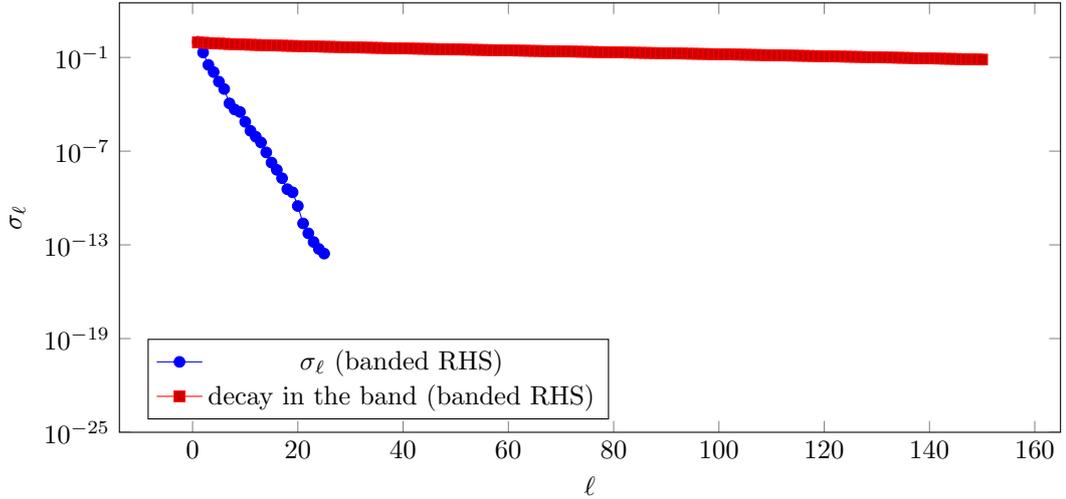
\begin{figure}\label{fig:test4}
                    \centering
                    \begin{tikzpicture}
                      \begin{semilogyaxis}[ylabel = $\sigma_\ell$ ,ymin=1e-25, xlabel = $\ell$, width = .95\linewidth,
                        height=.35\textheight, legend pos=south west]     
                        \addplot [skip coords between index={25}{150}, mark = *, color = blue]  table[x index = 0, y index = 6]  {test2-3.dat};
                        \addplot table[x index = 0, y index = 12] {test2-3.dat};
                        \legend{$\sigma_\ell$ (banded RHS), decay in the band (banded RHS)};
                      \end{semilogyaxis}
                    \end{tikzpicture}
                    \caption{We compute the solution $X\in\mathbb{R}^{n\times n}$, $n=300$, 
                    of $AX+XA = C$. $A=\trid(-1,2,1)$
                    (positive definite and ill-conditioned) while $C$ is a random diagonal matrix.}
                    \end{figure}
                   
                   \begin{itemize}
                    \item[\textbf{Test 1}] We compute $X_i$ as the solution of $AX_i+X_iA= C_i$ for $i=1,2$. 
                    The matrix $A$ is chosen symmetric tridiagonal with eigenvalues in $[0.2,+\infty)$, 
                    in particular $A$ is positive definite and well-conditioned. The right-hand side $C_1$ is taken
                    tridiagonal symmetric with random entries while $C_2$ is a random dense symmetric matrix with 
                    quasiseparable  rank $1$. 
                     In the first case, results from
                      \cite{haber, Palitta2017} ensure that --
                      numerically -- the banded structure is mantained
                      in the solution and this is shown in Figure
                      \ref{fig:test1}. Notice that the decay in the
                      off-diagonal singular values is much stronger than the
                      decay in the bandwidth so that, in this example,
                      it is more advantageous to look at the solution
                      as a quasiseparable matrix instead of a banded
                      one. Theorem \ref{thm:quasisep-sol-normal}
                      guarantees the solution to be quasiseparable
                      also in the second case whereas the banded
                      structure is completely lost.
                \item[\textbf{Test 2}] We compute the solution $X$ of
                  $AX+XA = C$. We consider
                  $A=\trid(-1,2,1)-1.99\cdot I$, so that it is
                  indefinite and ill-conditioned, and we set $C$ equal
                  to a random diagonal matrix. As highlighted in
                  Figure~\ref{fig:test2}, both the quasiseparable and
                  the band structure are not present in the solution $X$. 
                \item[\textbf{Test 3}] We compute $X_i$ as the
                  solution of $AX_i+X_iB= C_i$ for $i=1,2$.  The
                  matrix $A$ and $-B$ are chosen symmetric and
                  tridiagonal with eigenvalues in $[0.2,14]$ and
                  $[0.5,14]$, so both well conditioned but with
                  interlaced spectra. The right-hand side $C_1$ is
                  chosen tridiagonal symmetric while $C_2$ is set
                  equal to a random dense symmetric matrix with
                  quasiseparable rank $1$. The results in
                  Figure~\ref{fig:test3} suggest that both the
                  structures are preserved in the first case and lost
                  in the second case. Once again, in the case of
                  preservation, the decay in the off-diagonal singular
                  values is stronger than the decay in the bandwidth.
                   Notice that, when present, the
                    quasiseparability of the solution cannot be
                    predicted by means of Theorem
                    \ref{thm:quasisep-sol-normal}, but results from
                    \cite{haber,Palitta2017} can be employed to
                    estimate the banded structure of the
                    solution. This test shows how the banded structure
                    is a very particular instance of the more general
                    quasiseparable one.
                     \item[\textbf{Test 4}] We compute the solution $X$ of $AX+XA = C$. We chose $A=\trid(-1,2,1)$, 
                     so it is positive definite and ill-conditioned, and we set $C$ equal to a random diagonal matrix. 
                     Figure~\ref{fig:test4} clearly shows that quasiseparability is preserved while the banded structure
                     is not present in the solution $X$.
                      In this case, the quasiseparability of the solution can be shown by Theorem 
                     \ref{thm:quasisep-sol-normal}. Equivalently, one can exploits arguments in \cite{Palitta2017} where it 
                     has been shown that the solution can be represented as the sum of a banded matrix and a low-rank one so 
                     that $X$ is quasiseparable.
                    \end{itemize} 
                    To summarize, the situations where we know that the quasiseparable structure is present in the solution 
                    of \eqref{eq:lyap} are:
\begin{enumerate}[label=(\roman*)]
\item $A,B$ and $C$ quasiseparable and spectra of $A$ and $-B$ well
  separated\footnote{We consider the spectra to be well separated if
    Theorem~\ref{thm:quasisep-sol-normal} can be used to prove the
    quasiseparability. As we have seen, this also includes cases where
    the spectra are close, such as when they are separated by a
    line.};
\item $A,B$ and $C$ banded and well-conditioned.
\end{enumerate}  
On the other hand, for using
the computational approach of Section~\ref{sec:procedure} we need the spectra of $A$ and $-B$ to be separated by a line.

\section{HODLR-matrices}\label{sec:hodlr}
An efficient way to store and operate 
on matrices with an off-diagonal data-sparse structure is to use hierarchical formats. 
There is a vast literature on
this topic. See, e.g.,
\cite{hackbusch1999sparse,borm,hackbusch2000h} and the references
therein.
In this work, we rely on a particular subclass of the set of hierarchical representations sometimes called hierarchically 
off-diagonal low-rank (HODLR), which can be described as follows;
 let $A\in\mathbb{C}^{n\times n}$ be a $k$-quasiseparable matrix, we consider the $2\times 2$ block partitioning
 \[
   A=
   \begin{bmatrix}
     A_{11}&A_{22}\\
     A_{21}&A_{22}
   \end{bmatrix}
   , \qquad
   A_{11}\in\mathbb C^{n_1\times n_1}, \quad
   A_{22}\in\mathbb C^{n_2\times n_2} 
 \] where $n_1:=\lfloor \frac{n}{2} \rfloor $ and
 $n_2:=\lceil \frac{n}{2} \rceil$.  Since the antidiagonal blocks
 $A_{12}$ and $A_{21}$ do not involve any element of the main diagonal
 of $A$, they have rank at most $k$, so they are represented as
 low-rank outer products.  Then, the strategy is applied recursively
 on the diagonal blocks $A_{11}$ and $A_{22}$.  The process stops when
 the diagonal blocks reach a minimal dimension $n_{\text{min}}$, at
 which they are stored as full matrices. The procedure is graphically
 described in Figure~\ref{fig:Hmatrices}.  If $n_{\text{min}}$ and $k$
 are negligible with respect to $n$ then the storage cost is
 linear-polylogarithmic with respect to the size of the matrix, 
 as briefly summarized in Table~\ref{tab:complexity}. HODLR matrices
 are equivalent to hierarchical matrices with weak admissibility in the
 classification used in \cite{Hackbusch2016}.

It is natural to compare the storage required
by the HODLR representation and the truncation of banded structures, when they
are both present in the solution. 
Consider the following test: we compute
the solution $X$ of a Lyapunov equation
with a tridiagonal well-conditioned coefficient
matrix $A$ and a diagonal right hand-side
with random entries. As discussed in the previous section, the solution has a fast decay in the magnitude of the entries as we
get far from the main diagonal. 
We compare the accuracy obtained when the 
solution $X$ is stored in the HODLR format
 with different thresholds in the low-rank truncation of the off-diagonal blocks, and when a fixed number of diagonals 
 are memorized. 
In particular, the accuracy
achieved keeping $5k$ diagonals and truncating the SVD of the
off-diagonal blocks using thresholds 
$10^{-k}$, for $k = 0, \ldots, 16$, is illustrated in Figure~\ref{fig:memoryvsaccuracy}. 
We can see that the two approaches have comparable performances for this example. 
	The experiment is repeated using $A=\trid(-1,2,-1)$ highlighting the non feasibility of the sparse format in this case
as the banded structure is not preserved in the solution.

   \begin{figure}[!ht]
 \centering
 \includegraphics[width=0.18\textwidth]{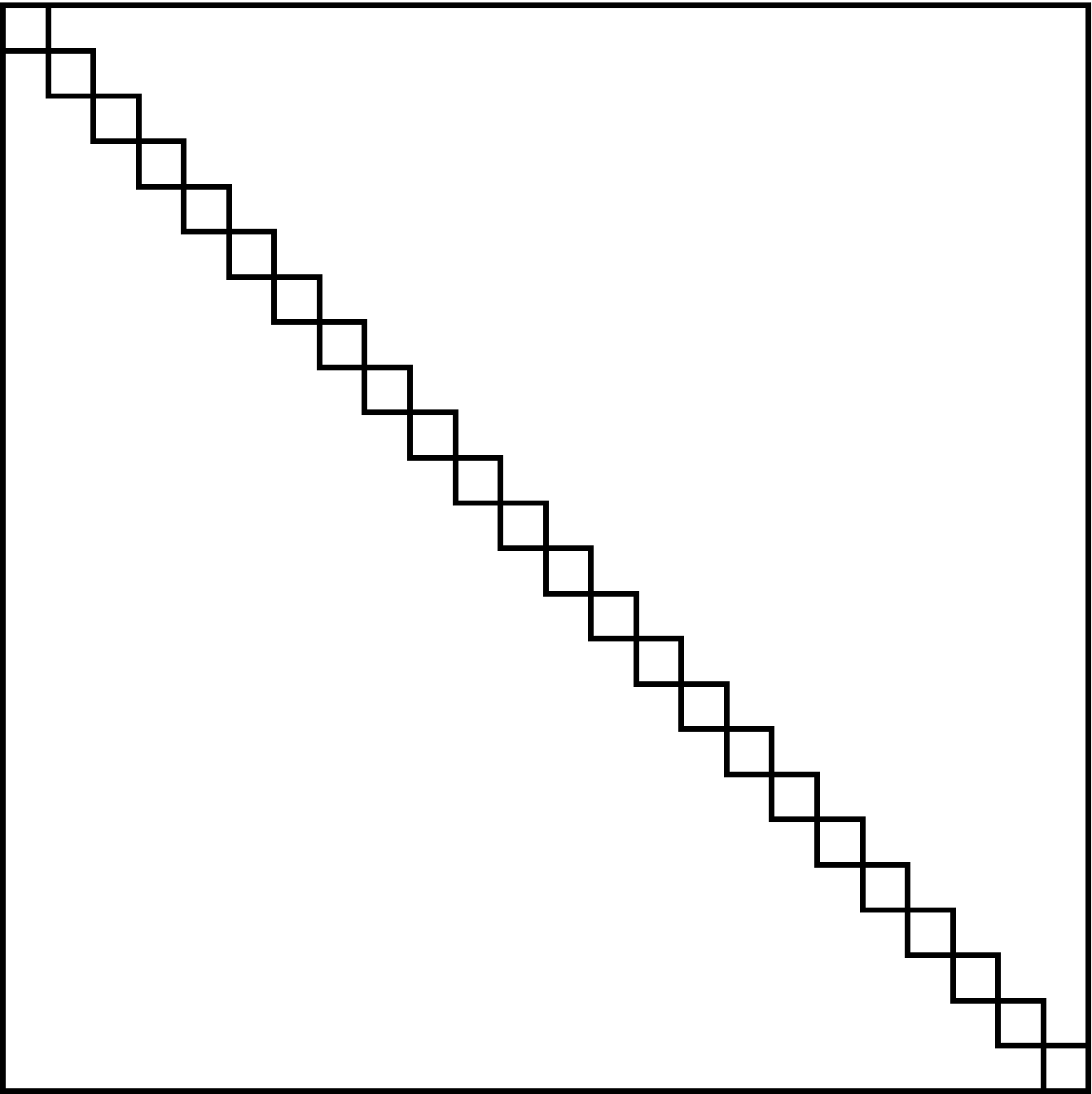}\qquad 
 \includegraphics[width=0.18\textwidth]{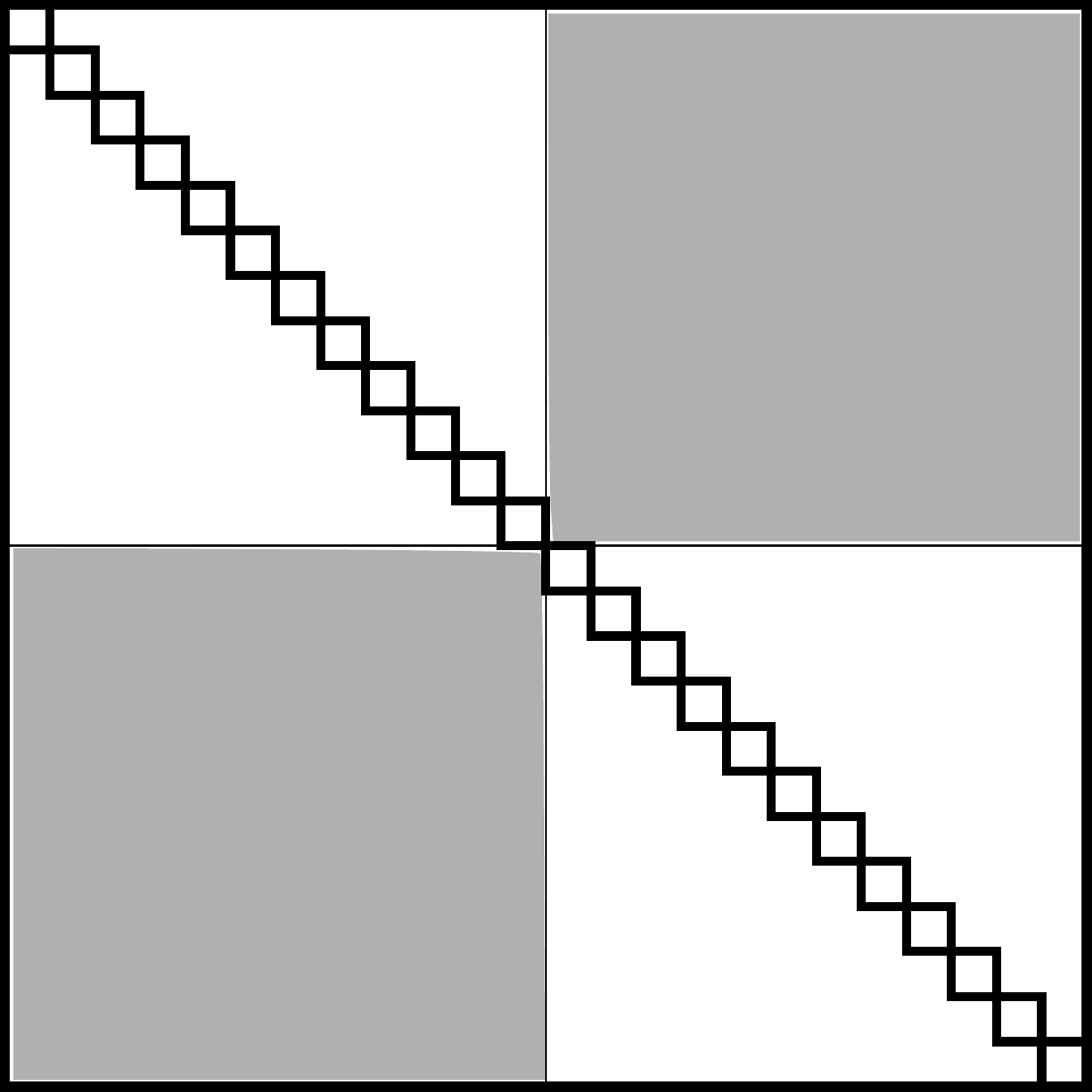}
 \qquad 
 \includegraphics[width=0.18\textwidth]{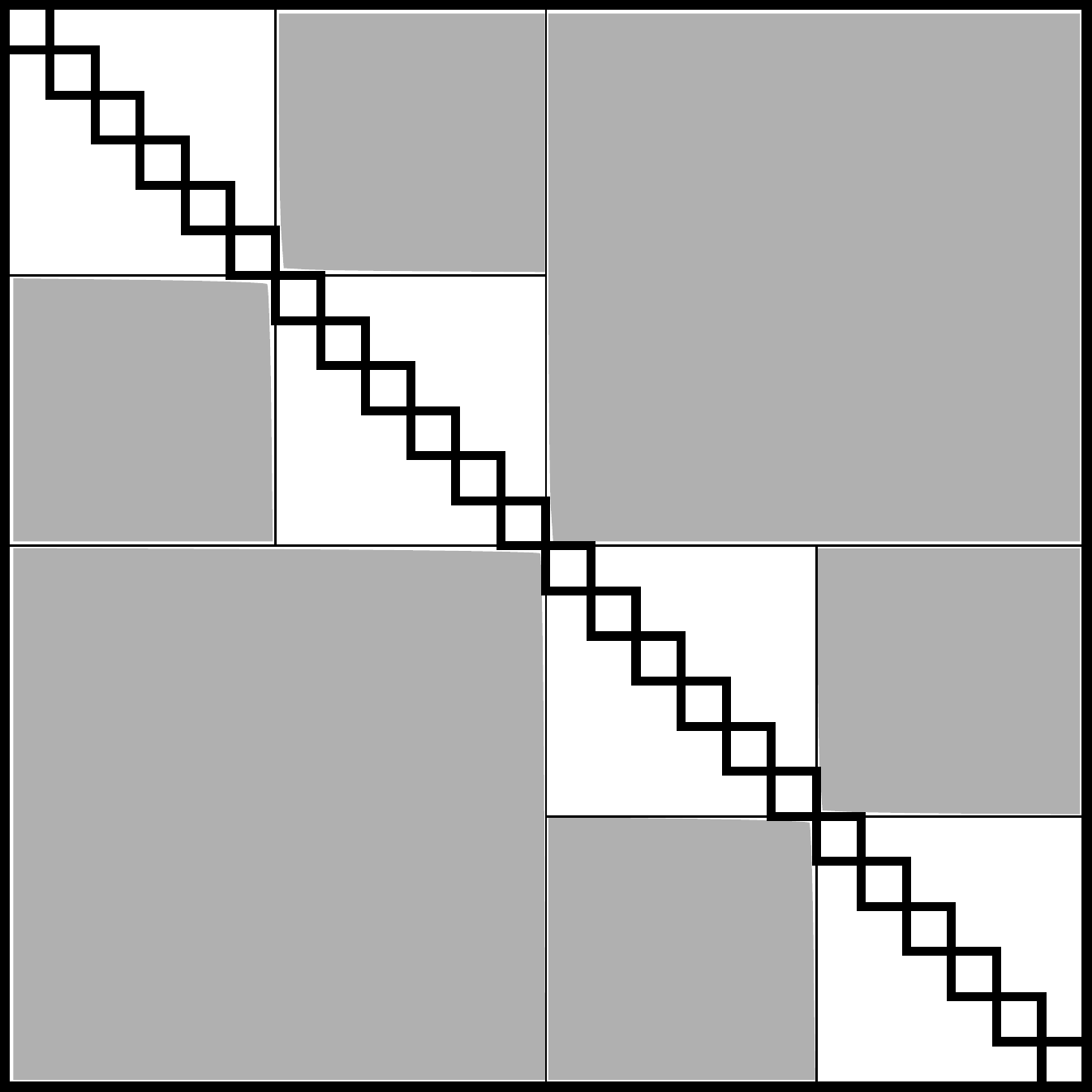}\qquad 
 \includegraphics[width=0.18\textwidth]{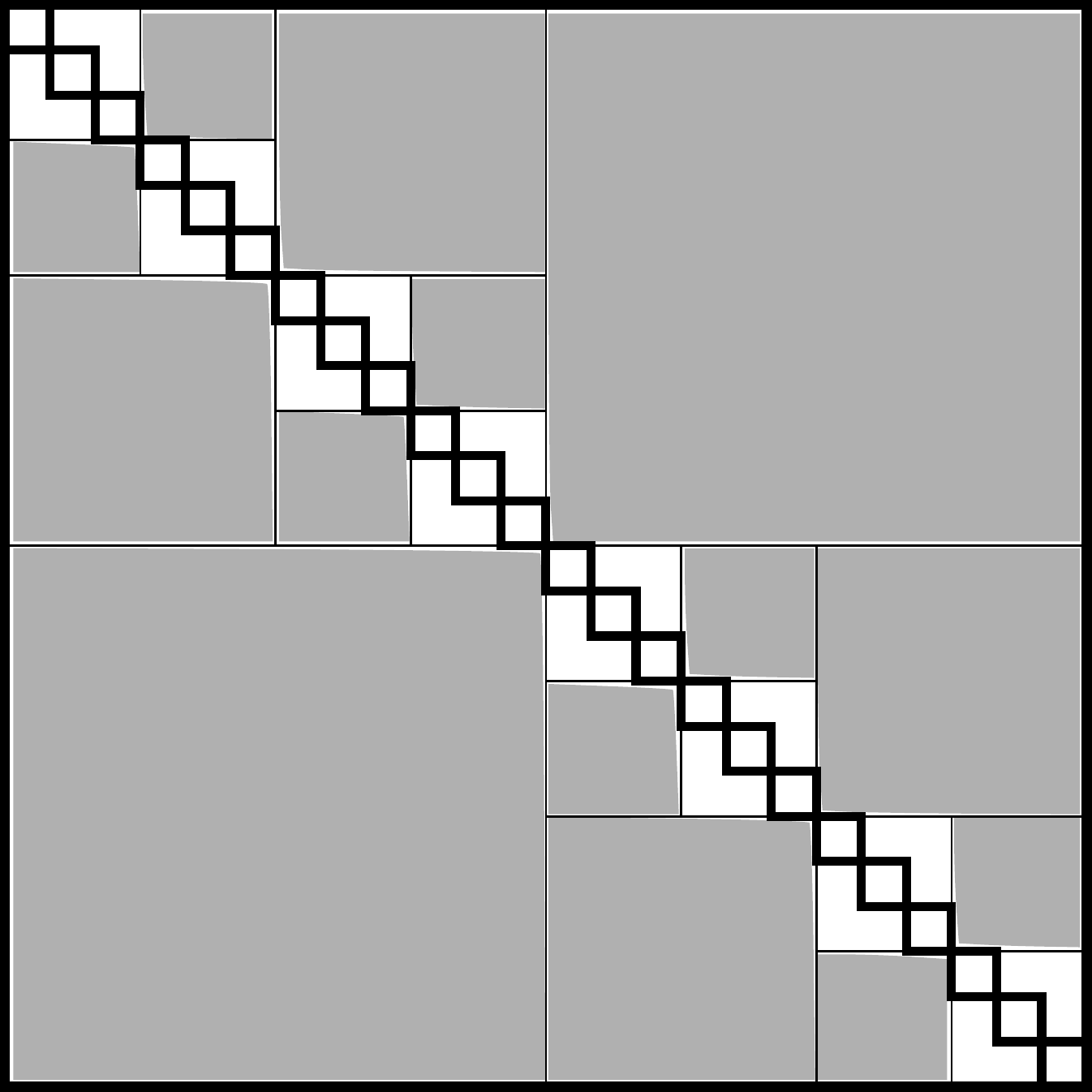}
  \caption{The behavior of the block partitioning in the HODLR-matrix representation. 
  The blocks filled with grey are low rank matrices represented in a compressed form, 
  and the diagonal blocks in the last step are stored as dense matrices.}\label{fig:Hmatrices}
 \end{figure}           

 The HODLR format has been studied intensively in the last decade and  algorithms  with  almost linear complexity for
 computing matrix operations are available, see, e.g., Chapter $3$ in \cite{Hackbusch2016}.   
 Intuitively, the convenience of using this representation in a procedure is strictly related with the growth of 
 the numerical rank of the off-diagonal blocks in the intermediate results. This can be formally justified with an 
 argument based on the Eckart-Young best approximation property, see Theorem 2.2 in \cite{Bini2017}.
 
 \begin{figure}
 \begin{center}
 \begin{tikzpicture}
 \begin{loglogaxis}[width=.45\textwidth, legend pos=south west,xlabel=Kilobytes, ylabel=$\norm{\widetilde X -X}_2/\norm{X}_2 $]
   \addplot table[x index=0, y index = 1]{mem-acc.dat};
   \addplot table[x index=2, y index = 3]{mem-acc.dat};
   \legend{Sparse, HODLR};
 \end{loglogaxis}           
 \end{tikzpicture} 
 \quad
 \begin{tikzpicture}
  \begin{loglogaxis}[width=.45\textwidth, legend pos=south west,xlabel=Kilobytes, ylabel=$\norm{\widetilde X -X}_2/\norm{X}_2 $]
    \addplot table[x index=0, y index = 1]{mem-acc-lap.dat};
    \addplot table[x index=2, y index = 3]{mem-acc-lap.dat};
    \legend{Sparse, HODLR};
  \end{loglogaxis}           
  \end{tikzpicture} 
  \end{center}  
  \caption{We show the accuracy obtained approximating the solution $X$ to a 
  Lyapunov equation keeping a certain number of diagonals and by truncating
  the HODLR representations with $n_{min}=50$. The plot reports
  the accuracy obtained with respect to the memory consumption when $A$ is banded and well conditioned (left), and for $A=\trid(-1,2,-1)$ (right). The matrices have dimension $n=2048$, the storage cost for the dense matrix $X$ is $32678$ KB.}
  \label{fig:memoryvsaccuracy}   
 \end{figure}
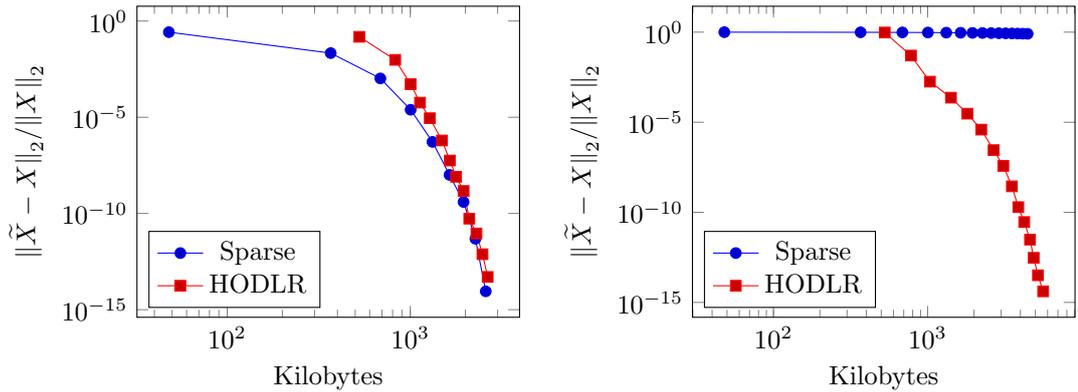

We relied on \texttt{hm-toolbox} for our experiments, which is available 
at \url{https://github.com/numpi/hm-toolbox}, and implements HODLR arithmetic.

\begin{table}
                                \begin{center}                   
                                
                                {
                 \begin{tabular}{cc}
                  \toprule
                   Operation & Computational complexity \\ \midrule
                  Matrix-vector multiplication & $O(k n\log(n))$\\
                  Matrix-matrix addition & $O(k^2 n\log(n))$\\
                  Matrix-matrix multiplication & $O(k^2 n\log^2(n))$\\
                  Matrix-inversion & $O(k^2 n\log^2(n))$\\
                  Solve linear system & $O(k^2 n\log^2(n))$\\
                   \hline
            
                 \end{tabular}
                }
                \end{center}
                \caption{Computational complexity of the HODLR-matrix arithmetic. The integer $k$ is the maximum of the quasiseparable ranks of the inputs while $n$ is the the size of the matrices. }\label{tab:complexity}
                \end{table}
                
\section{Solving the Sylvester equation}\label{sec:procedure}
In this section we show how to deal with the issue of solving
\eqref{eq:lyap} taking advantage of the quasiseparable structure of the data.
We first discuss the matrix sign iteration and then we show how to efficiently evaluate the integral 
formula~\eqref{eq:integral}. Both these algorithms are implemented in the \texttt{hm-toolbox}.

\subsection{Matrix Sign Function}\label{Sec_signfun}
Here, we briefly recall the Matrix Sign Function iteration, first proposed 
in the  $\mathcal{H}$-format by Grasedyck, Hackbusch and Khoromskij in \cite{grasedyck-riccati}. 
We use HODLR arithmetic in the iteration scheme proposed by Robert in 
\cite{roberts-sign}, that relies on the following result. 

\begin{theorem}
Let $A,B\in\mathbb C^{n\times n}$ be positive definite, then the solution $X$ of \eqref{eq:lyap} verifies
\begin{equation}
X=\frac 12 N_{12}, \qquad 
\begin{bmatrix}
N_{11}&N_{12}\\
0&N_{22}
\end{bmatrix}:=\sign\left(
\begin{bmatrix}
A&C\\
0&-B
\end{bmatrix}\right), 
\end{equation}
and --- given a square matrix $M$ --- we define $\sign(M):=\frac{1}{\pi\mathbf i}\int_{\gamma}(zI-M)^{-1}dz$ with $\gamma$ a path 
of index $1$ around the eigenvalues of $M$ with positive real part.
\end{theorem}
The sign function of a square matrix $S:=\sign(M)$ can be approximated applying the Newton's method to the equation $X^2-I=0$ 
with starting point $S_0=M$. This requires to compute the sequence
\begin{equation}\label{sign-sequence}
S_0=M,\qquad S_{i+1}=\frac 12(S_i+S_i^{-1}),
\end{equation}
which converges to $S$, provided that $M$ has no eigenvalues on the imaginary axis \cite{grasedyck-riccati}. Rewriting
\eqref{sign-sequence} block-wise yields
\begin{equation}\label{block-sign-sequence}
A_{i+1}= \frac 12(A_i+A_i^{-1}),\quad B_{i+1}= \frac 12(B_i+B_i^{-1}),\quad C_{i+1}=\frac 12(A_i^{-1}C_iB_i^{-1}+C_i),
\end{equation}
where $A_0=A, B_0=B$, $C_0=C$ and $C_{i+1}\rightarrow 2X$. As stopping criterion we used the condition 
\[
\norm{A_{i+1}-A_i}_F+\norm{B_{i+1}-B_i}_F+\norm{C_{i+1}-C_i}_F\leq \sqrt{\epsilon},
\]
where $\epsilon$ is the selected accuracy.  This can be heuristically
justified saying that since Newton is quadratically convergent, if the
above quantity is a good estimate of the error of the previous step
then we have already obtained the solution at the required precision.

We implemented the algorithm in \cite{grasedyck-riccati}, that performs the
iteration using  hierarchical matrix arithmetic. 
When an appropriate scaling of $A$ and $B$ is performed, convergence is
reached in few steps \cite{Higham2008}. The scaling strategy
is crucial to keep the number of iterations of the Newton scheme low, 
and the scaling parameter $\alpha>0$ can be optimally chosen at
every iteration, as shown in \cite{Higham2008}. When the spectra of $A$ and $B$ are real, the optimal choice is
$\alpha_i=\sqrt{\norm{S_i^{-1}}_2/\norm{S_i}_2}$. However, if hierarchical matrix arithmetic is employed,
the scaling strategy may introduce a non-negligible error
propagation as outlined in \cite{grasedyck-riccati}.
We found out that a good 
trade-off is to scale only in the first iteration. This does
not affect the accuracy of the iterative steps if the matrix $S_0$ can
be exactly represented in the hierarchical format \cite[Remark
5.3]{grasedyck-riccati}, and allows to keep the number of
iterations proportional to $\log(\max\{\kappa (A),\kappa(B)\})$
\cite{grasedyck-riccati}. For instance, in the case of $A=B$ being the
discrete Laplacian operator, which has a condition number that grows
as $\mathcal O(n^2)$, the latter choice makes the computational cost
of the approach $\mathcal O(n\log^3(n))$.

\subsection{Solution by means of the integral formula}\label{Sec_integral}

We now propose to apply a quadrature scheme for evaluating the
semi-infinite integral in \eqref{eq:integral}.  We perform the change of
variable $x=f(\theta):= L\cdot\cot\left(\frac{\theta}{2}\right)^2$
where $\theta$ is the new variable and $L$ is a parameter 
chosen to optimize the convergence. This is a very common strategy
for the approximation of integral over infinite domain, which is
discussed in detail by Boyd in \cite{boyd1987exponentially}. 
We transform \eqref{eq:integral} into
\begin{equation}\label{eq:integral2}
X=2L\int_0^{\pi}\frac{\sin(\theta)}{(1-\cos(\theta))^2} \ e^{-Af(\theta)}Ce^{-Bf(\theta)}d\theta,
\end{equation}
which can be approximated by a Gauss-Legendre quadrature scheme.
Other quadrature formulas, as Clenshaw-Curtis rules, can be employed. However,
as discussed by Trefethen in
\cite{trefethen2008gauss}, the difference between Gauss-Legendre and Clenshaw-Curtis formulas is
small. Moreover, in most of our tests, Gauss-Legendre schemes showed some slight computational advantages 
over Clenshaw-Curtis rules as the cost of computing the integration points is
negligible%
\footnote{In practice we have precomputed the points for the usual
  cases, so that an explicit computation of them is never carried out
  in the numerical experiments.}.

The quadrature scheme yields an approximation of \eqref{eq:integral2}
of the form
\begin{equation}
  \label{eq:quadrature-rule}
  X\approx\sum_{j=1}^m\omega_j\cdot e^{-Af(\theta_j)}Ce^{-Bf(\theta_j)},    
\end{equation}
where $\theta_j$ are the Legendre points,
$\omega_j=2 L w_j \cdot \frac{\sin(\theta_j)}{ (1 - \cos(\theta_j))^2
}$ and $w_j$ are the Legendre weights.

Finally, we numerically approximate the quantities
$e^{-Af(\theta_j)}$ and $e^{-Bf(\theta_j)}$, which represents the dominant cost of the algorithm.
For this task, we have
investigated two rational approximations, which have been
implemented in our toolbox.

\begin{description}
\item[Padé] The matrix exponential $e^A$ can be well approximated by a
  diagonal Padé approximant of degree $(d, d)$ if $\norm{A}$ is small
  enough\footnote{The exact choice of the ball where Padé is accurate enough
    depends on the desired accuracy and the value of $d$.}.
  We thus satisfy this condition by using the relation
  $e^{A} = (e^{2^{-k} A})^{2^k}$, a tecnhique typically called
  ``scaling and squaring''. The Padé approximant is known explicitly
  for all $d$.  See, e.g., \cite[Chapter 10]{Higham2008}.  In this
  case the evaluation of the matrix exponential requires $2d + 3 + k$
  matrix multiplication and one inversion where
  $k = \lceil \log_2 \norm{A} \rceil$.  This strategy is also
  implemented in the MATLAB function \texttt{expm}.
\item[Chebyshev] Since $A$ is supposed to be positive definite, the matrix exponential $e^{-tA}$ can 
be approximated by a rational Chebyshev function that is uniformly accurate for every positive
  value of $t$, as described by Popolizio and Simoncini in
  \cite{popolizio2008acceleration}. The rational function
  is of the form
    \[
    e^{x} \approx \frac{r_1}{x - s_1} + \ldots + \frac{r_d}{x - s_d}. 
  \]
  Given the poles and the weights in the above expansion,
  this strategy requires $d$ inversions and additions. 
  See, e.g., \cite{Palitta2017} for a numerical procedure to compute the poles and weights $s_i$, $r_i$.
   
\end{description}
\begin{remark}\label{rem:pade}
  In our tests, evaluating the matrix exponential $e^{-f(\theta)A}$ by
  means of the Pad\'e approximant performs better when $f(\theta)A$ has a
  moderate norm. When $f(\theta)\norm{A}_2$ is large the squaring phase becomes
  the bottleneck of the computation. In this case we rely
  on the rational Chebyshev expansion, which has a cost independent of
  $\norm{A}_2$.
\end{remark}

The procedure is summarized in
Algorithm~\ref{alg:sylv-integral}.
The evaluations of the matrix exponentials 
$\Call{expm}{-f \cdot A}$, $\Call{expm}{-f \cdot B}$
are performed according to the strategy outlined in Remark~\ref{rem:pade}.
\begin{algorithm} 
  \begin{algorithmic}[1]
    \Procedure{lyap\_integral}{$A,B,C,m$} \Comment{Solves $AX + XB = C$ with $m$ integration points}
    \State $L \gets 100$ \Comment{This can be tuned to optimize the accuracy}
    \State $[w,\theta] \gets \Call{GaussLegendrePts}{m}$ \Comment{Integration points and weights on $[0,\pi]$}
    \State $X \gets 0_{n \times n}$
    \For{$i = 1, \ldots, m$}
      \State $f \gets L \cdot \cot(\frac{\theta_i}{2})^2$
      \State $X \gets X + w_i \frac{\sin(\theta_i)}{(1 - \cos \theta_i)^2} \cdot
        \Call{expm}{-f \cdot A} \cdot C \cdot \Call{expm}{-f \cdot B}$
        \EndFor
        \State $X \gets 2L \cdot X$
    \EndProcedure
  \end{algorithmic}
  \caption{Solution of a Sylvester equation by means of the integral formula}
  \label{alg:sylv-integral}  
\end{algorithm}

\section{Solving certain generalized equations}\label{Sec_genLyap}
The solution of certain generalized Sylvester equations can be
recast in terms of standard Sylvester ones. The results in
Section~\ref{sec:quasisep-solution} thus suggest the presence of a
quasiseparable structure also in the solution of this kind of
equations. For the  sake of simplicity, we focus on generalized Lyapunov
equations, but the approach we are going to present can be easily
extended to the Sylvester case as well. We consider equations of the form
\begin{equation}\label{eq:gen-lyap}
AX+XA+\sum_{j=1}^sM_jXM_j^T= C,\qquad A,X,C,M_j\in\mathbb R^{n\times n},
\end{equation}
where $A$ is symmetric positive definite, both $A$ and $C$ are quasiseparable and $M_j$ is low rank for $j=1,\dots,s$. 
 We generalize Theorem~\ref{thm:quasisep-sol-normal} to this framework.
\begin{corollary} \label{cor:sylv-low-rank}
 Let $A$ be a symmetric positive definite matrix of quasiseparable
  rank $k_A$ and let $\kappa_A$ be its condition number. Moreover,
  consider the generalized Lyapunov equation $AX + XA + \sum_{j=1}^sM_jXM_j^T = C$, with $M_j$ of rank $r_j$, $j=1,\dots,s$, 
  and $C$ of
  quasiseparable rank $k_C$. Then a generic off-diagonal block $Y$ of
  the solution $X$ satisfies
  \[
    \frac{\sigma_{1 + k\ell}(Y)}{\sigma_1(Y)} \leq 4\rho^{-2\ell},
  \]
  where $k:=2k_A+k_C+\sum_{j=1}^sr_j$,
 $\rho=\operatorname{exp}\left(\frac{\pi^2}{2\mu(\kappa_A)}\right)$ and $\mu(\cdot)$ is defined as in Lemma~\ref{lem:sylv-bound}.
\end{corollary}
\begin{proof}
The solution $X$ satisfies $AX + XA = C-\sum_{j=1}^sM_jXM_j^T $ where the right hand side has quasiseparable rank $k_C+\sum_{j=1}^sr_j$. By applying Theorem~\ref{thm:quasisep-sol-normal} to the latter we get the claim.
\end{proof}
Equation \eqref{eq:gen-lyap} can be rephrased as a $n^2\times n^2$ linear system by Kronecker transformations 
\[
\left(\mathcal L+\mathcal M\right)\vect(X)=\vect(C),\qquad \mathcal L:=I\otimes A+A\otimes I,\qquad\mathcal M:= \sum_{j=1}^sM_j\otimes M_j.
\]
We assume that $\mathcal L$ is invertible, 
and such that $M_j=U_jV_j^T$ with $U_j,V_j\in\mathbb{R}^{n\times r_j}$, $j=1,\dots,s$. 
In particular, the matrix $\mathcal M\in\mathbb R^{n^2\times n^2}$ is of rank $r:=\sum_{j=1}^sr_j^2$ and it can be factorized as
\[
\mathcal M=UV^T=\left[\begin{array}{c|c|c}
&&\\
U_1\otimes U_1&\dots&U_s\otimes U_s\\
&&
\end{array}\right]\cdot
\left[\begin{array}{c|c|c}
&&\\
V_1\otimes V_1&\dots&V_s\otimes V_s\\
&&
\end{array}\right]^T.
\]
Plugging this factorization into the Sherman-Morrison-Woodbury formula we get
\begin{equation}\label{SMW}
\vect (X) = \mathcal L^{-1} \vect (C)- \mathcal L^{-1}U\left(I_r+V^T\mathcal L^{-1}U\right)^{-1}V^T\mathcal L^{-1}\vect (C).
\end{equation}
See, e.g., \cite{Benner:2013:low}. As shown in \cite[Section 4]{Ringh2016},
the solution of \eqref{eq:gen-lyap} by \eqref{SMW} requires the inversion of a $r\times r$ linear system and the solution
of $r+1$ Lyapunov equations of the form
$$AZ+ZA=C,\quad AZ_{ij}+Z_{ij}A=\widetilde U_{ij}, \mbox{ for } i=1,\dots,s, j=1,\dots,r_i^2,$$
where $\widetilde U_{ij}=\vect^{-1}((U_i\otimes U_i) (:,j))$ has rank
1. Since $A$ and $C$ are quasiseparable, the matrix $Z$ can be
computed by one of the method presented in the previous sections,
whereas well-established low-rank methods can be employed in computing
$Z_{ij}$'s. In our tests we have used the method based on extended
Krylov subspaces discussed in \cite{Simoncini2007}. 
The procedure is illustrated in Algorithm~\ref{alg:swm}.

       \begin{algorithm}\caption{Solution of a generalized Lyapunov equation (low-rank $\mathcal{M}$) 
       by \eqref{SMW}}\label{alg:swm}
       \begin{algorithmic}[1]
       \Procedure{SMW\_Gen\_Lyap}{$A,C,U_i,V_i$}\Comment{Solve $AX+XA+\sum_{i=i}^sU_iV_i^TXV_iU_i^T=C$ }
       \State $\widehat{X}\gets $ $A\widehat{X}+\widehat{X}A=C$ \label{first_sol}
       \For{$h=1:s$}
       \State $\widehat X_h\gets V_h^T \widehat X V_h$
       \EndFor
       \For{$h=1:s$, $i,j=1:r_h$}
       \State $ U_{ij}^h\gets U_h(:,i)U_h(:,j)^T$
       \State $Z_k\gets$ $ AZ_k+Z_kA= U_{ij}^h$ \Comment{$k:=j+(i-1)r_h+\sum_{t=1}^{h-1}r_t^2$ }
	   \For{$m=1:s$}
       \State $ W_{mk} \gets V_m^TZ_kV_m$
       \EndFor 
       \State $[Z_{1+\sum_{t=1}^{h-1}r_t^2},\dots,Z_{\sum_{t=1}^{h}r_t^2}]=Z_h^{(u)}Z_h^{(v)T}$
       \EndFor
       \State $R\gets \left(I_r+\Big[\vect(W_{mk})\Big]_{\substack{
           m=1,\dots,s \\ k=1,\ldots,r
           }}\right)^{-1}$
        \State $\widehat Z \gets R\cdot[ \vect(\widehat{X}_1);\dots;\vect(\widehat X_s)]$
        
        \State  $S\gets \sum_{h=1}^sZ_h^{(u)}\cdot \widehat Z_{r_h}\cdot Z_h^{(v)T}$
        \Comment{$\widehat Z_{r_h}:= \texttt{reshape}(\widehat Z(1+\sum_{i=1}^{h-1}r_i^2:\sum_{i=1}^hr_i^2),r_h,r_h)$} 
        
        \State $X\gets \widehat X - S$
       \State \Return $X$
       \EndProcedure
       \end{algorithmic}
       \end{algorithm}

Another interesting class of generalized Lyapunov equations consists of equation \eqref{eq:gen-lyap} with 
 $\rho(\mathcal{L}^{-1}\mathcal M)<1$, where $\rho(\cdot)$ denotes the spectral radius.
For this kind of problems, the matrices $M_j$ do not need to be low-rank, but we suppose
they all have a small quasiseparable rank.
In this case, one can consider the Neumann series expansion of 
 $(\mathcal L+\mathcal M)^{-1}$ as done in \cite{palitta-svezia}. More precisely, it holds
\[
(\mathcal L+\mathcal M)^{-1}= (I+\mathcal L^{-1}\mathcal M)^{-1}\mathcal L^{-1}=
\sum_{j=0}^{\infty} (-1)^j(\mathcal L^{-1}\mathcal M)^j\mathcal L^{-1},
\]
so that the solution $X$ to \eqref{eq:gen-lyap} verifies
\begin{equation}\label{eq:neuman}
X=\sum_{i=0}^\infty Z_i,\quad\text{where}\quad \begin{cases}
AZ_0+Z_0A=C, \\
AZ_{i+1}+Z_{i+1}A= -\sum_{j=1}^sM_jZ_iM_j^T.\\
\end{cases}
\end{equation}
A numerical solution can thus be computed truncating the series in \eqref{eq:neuman}, that is 
$X\approx X_\ell:=\sum_{i=0}^\ell Z_i$ where the number of terms $\ell$ is related to the accuracy of the 
computed solution.
If $\ell$ is moderate, that is $\rho(\mathcal{L}^{-1}\mathcal M)\ll1$, $X_\ell$ is the sum of few quasiseparable
matrices $Z_i$ and it is thus quasiseparable.
Notice that the quasiseparability of $M_j$'s
is necessary to mantain a quasiseparable structure in the right-hand sides $-\sum_{j=1}^sM_jZ_iM_j^T$, $i=0,\ldots,\ell-1$.

\section{Numerical Experiments}
\label{sec:numerical}
An extensive computational comparison among different approaches 
for quasiseparable Sylvester equations -- as well
as their implementation -- is still lacking in the literature,
and in this section we perform some numerical experiments trying to fill this gap.
To this end, we employ the MATLAB
\texttt{hm-toolbox} that we have developed while writing this paper. 
The toolbox --- which includes all the tested algorithms --- is now freely available
at \url{https://github.com/numpi/hm-toolbox}. 
All the timings reported are relative
to MATLAB 2016a run on a machine with a CPU running at 3066 MHz, $12$ cores\footnote{All the available cores have only be used to run the parallel implementation of the solver based on the integral formula. All the other solvers did not exploit the parallelism in the machine.}, and 
192GB of RAM.

Each of the following sections contains a specific example. Some of these
are artificially constructed to describe particular cases; others present
real or realistic applications, arising from PDEs. We start by describing
the classical Laplacian case, and then proceed comparing our results with
a 2D heat equation arising from practical applications. Eventually, we show how to solve some partial 
integro-differential equations.

To test the accuracy of our approach we report the relative residual on the linearized system of the computed solution.
If $\mathcal S$ is the coefficient matrix of the linearized system we measure the relative residual, 
\[
r(\mathcal S,X):= \frac{ \norm{\mathcal S\cdot x -c}_2 }{\norm{\mathcal S}_F\cdot \norm{x}_2},\qquad x=\vect(X),
\qquad c=\vect(C),
\]
which can be easily shown to be the relative backward error in the Frobenius norm \cite{higham2002accuracy}.
When we deal with (standard) Sylvester problems, we have $\mathcal S= I\otimes A+B\otimes I$ with $A$ and $B$ symmetric.
This allows to use the --- easier to compute --- bound
\[
\norm{\mathcal S}_F^2\geq  n(\norm{A}_F^2+\norm{B}_F^2),
\]
so that 
\[
r(\mathcal S,X)= \frac{ \norm{\mathcal S\cdot x -c}_2 }{\norm{\mathcal S}_F\cdot \norm{x}_2}
=\frac{ \norm{AX+XB-C}_F }{\norm{I\otimes A+B\otimes I}_F\cdot \norm{X}_F}\leq
\frac{ \norm{AX+XB-C}_F }{\sqrt{n(\norm{A}_F^2+\norm{B}_F^2)}\cdot \norm{X}_F},
\]
and we actually compute and check the right-hand side in the above expression.

In case of a generalized Lyapunov equation, the system matrix is of the form 
$\mathcal S:= I\otimes A+B\otimes I+ M\otimes M$ and 
the relative residual norm is bounded using the inequality 
$\norm{\mathcal S}_F\geq \norm{I\otimes A+B\otimes I}_F-\norm{M}_F^2$.
Notice that this never requires to form the large system matrix
$\mathcal S$, and can be evaluated using the arithmetic of
hierarchical matrices when considering large scale problems. 

\subsection{The Laplace equation}\label{Laplace_ex}

We consider the 2-dimensional (2D) Laplace equation on the unit square $\Omega=[0,1]^2$

\[
  \begin{cases}    
    - \Delta u =  f(x, y) & (x,y) \in \Omega, \\
    u(x,y) = 0 & (x,y) \in \partial \Omega.
  \end{cases}, \qquad
  \Delta u = \frac{\partial^2 u}{\partial x^2} +
    \frac{\partial^2 u}{\partial y^2}.
\]

We construct the matrix $A$ representing the finite difference
discretization of the second-order derivative in the above equation on
a $n \times n$ grid using centred differences,
so that we obtain the equation $AX + XA = C$,
with
\[
  A = \frac{1}{h^2}
  \begin{bmatrix}
    2 & -1 \\
    -1 & \ddots & \ddots \\
    & \ddots & \ddots & -1 \\
    & & -1 & 2 \\
  \end{bmatrix},  \qquad
  h = \frac{1}{n - 1},
\]
and $C$ contains the samplings of the function $f(x, y)$ on our
grid. We consider the case where $f(x, y) = \log(1+|x-y|)$. As already
discussed, the latter choice provides a right hand-side which is
numerically quasiseparable.  This is due to the fact that in the
sub-domains corresponding to the off-diagonal blocks, $f$ is analytic
and it is well approximated by a sum of few separable functions.  One
can also exploit this property in order to retrieve the HODLR
representation of $C$; the sampling of a separable function
$g(x)\cdot h(y)$ on a square grid provides a matrix of rank $1$ and
the sampling of $g$ and $h$ yield its generating factors. The
computation of the expansion of $f$ in the sub-domains has been
performed by means of Chebfun2 \cite{cheb}.

Using \texttt{hm-toolbox}, the equation can be solved with a few
MATLAB instructions, as shown in Figure~\ref{fig:matlab-laplacian} for
the case $n = 2048$. The function \texttt{hmoption} can be used to set
some options for the toolbox. In this case we set the relative
threshold for the off-diagonal truncation to $10^{-12}$, and the
minimum size of the blocks to $256$.
 The class \texttt{hm} implements the hierarchical structure,
 and here we initialize it using a sparse tridiagonal matrix.
 Invoking \texttt{lyap} function uses our implementation
 specialized for $\mathcal H$-matrices.
 \begin{figure}
   \begin{framed}
\noindent\texttt{>> n = 2048;} \\
\noindent\texttt{>> hmoption('threshold', 1e-12);} \\
\noindent\texttt{>> hmoption('block-size', 256);} \\
\noindent\texttt{>> f = @(x,y) log(1 + abs(x - y));} \\
\noindent\texttt{>> A = (n-1)\textasciicircum 2 * spdiags(ones(n,1) * [ -1 2 -1 ], -1:1, n, n); } \\
\noindent\texttt{>> H = hm('tridiagonal', A);} \\
\noindent\texttt{>> C = hm('chebfun2', f, [-1,1], [-1,1], n);} \\
\noindent\texttt{>> X = lyap(H, C, 'method', 'sign');} \\
\noindent\texttt{>> hmrank(X)} \\[5pt]
\noindent\texttt{ ans = } \\[5pt]
\noindent\phantom{aaaaa}\texttt{     13}     
   \end{framed}        
   \caption{Example MATLAB session where the {\texttt{\textup{hm-toolbox}}} is
     used to compute the solution of a Lyapunov equation
     involving the 2D Laplacian and a numerically quasiseparable
   right hand-side.}
   \label{fig:matlab-laplacian}
   \end{figure}
In this example, we used the sign iteration, which is the default
method for the implementation of \texttt{lyap}. 
The quasiseparable rank of the solution (obtained
using the function \texttt{hmrank})
is $13$, which is reasonably small compared to the size of the problem. 

\begin{table}[t]
  \centering
  \small
  \begin{filecontents*}{experiment1.dat}
512	0.70967	2.9733e-12	13	3.6942	3.9212e-09	13	0.85298 1.9373e-09      5.9276e-08      1.5151
1024	1.7277	4.3283e-12	14	9.3704	8.7064e-10	14	7.521   2.1895e-08      2.6316e-08      3.2074
2048	4.7609	2.0271e-11	13	22.779	7.2132e-10	14	80.149  5.1967e-08      2.6356e-08      6.3413
4096	13.331	5.1888e-11	15	57.154	5.7338e-11	12	523.16  5.9939e-07      1.0782e-08      14.511
8192	35.93	3.646e-11	13	136.42	9.2336e-12	11	inf	0	0	31.821
16384	92.834	1.0037e-10	14	334.75	3.1383e-12	11	inf	0	0	70.277
32768	245.82	1.5487e-10	16	790.28	1.4245e-12	11	inf	0	0	154.65
65536	609.86	1.3345e-10	15	1825.20	8.8606e-13	10	inf	0	0	351.82	
1.3107e+05	1474.56	1.5756e-10	17	4122.17	2.0254e-12	9	inf	0	0	763.05	
\end{filecontents*}
\pgfplotstabletypeset[%
    columns={0,1,2,3,4,10,5,6,7},
    columns/0/.style={column name=$n$},
    columns/1/.style={column name=$T_{\text{Sign}}$},
    columns/2/.style={column name=$Res_{\text{Sign}}$},    
    columns/3/.style={column name=QS rk},
    columns/4/.style={column name=$T_{\text{Exp}}$},
    columns/10/.style={column name=$T_{\text{ParExp}}$},
    columns/5/.style={column name=$Res_{\text{Exp}}$},
    columns/6/.style={column name=QS rk},
    columns/7/.style={column name=$T_{\text{lyap}}$}
  ]{experiment1.dat}
  \caption{Timings and features of the solution of the Laplacian equation for
    different grid sizes. For the methods based on the HODLR arithmetic the minimum block size is set to $256$ and the relative threshold in truncation is $\epsilon=10^{-12}$. For small problems we also report the timings
    of the \texttt{lyap} function included in the Control Toolbox in
    MATLAB. The relative residuals of the Lyapunov equation are reported as well for the
    different methods. The residuals for the parallel version of the
      method based on the exponential have been omitted since they coincide 
      with the ones of the sequential one. In fact, the two algorithms perform
      exactly the same computations.}
  \label{tab:laplacian}
\end{table}

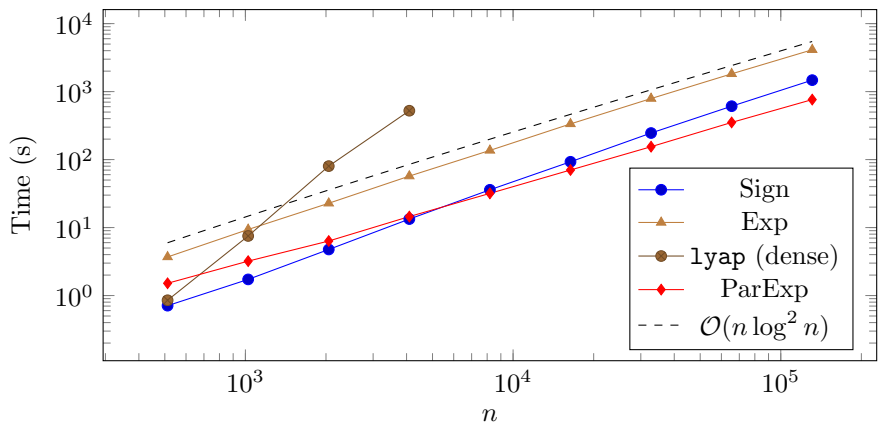
\begin{figure}
  \centering

  \begin{tikzpicture}
    \begin{loglogaxis}[width=.8\linewidth, height=.3\textheight,
      legend pos = south east, ymin=1.1e-1,
      xlabel = $n$, ylabel = Time (s)]
      \addplot table[x index=0, y index=1]  {experiment1.dat};
      \addplot[brown,mark=triangle*] table[x index=0, y index=4]  {experiment1.dat};
      \addplot table[x index=0, y index=7] {experiment1.dat};
      \addplot[red,mark=diamond*] table[x index=0, y index=10] {experiment1.dat};
      \addplot[domain=512:131000,dashed] {3e-4 * x * ln(x) * ln(x)};
      \legend{Sign, Exp, \texttt{lyap} (dense), ParExp, $\mathcal O(n \log^2 n)$};
    \end{loglogaxis}
  \end{tikzpicture}
  
  \caption{Timings for the solution of the Laplacian equation for
    different grid sizes. The performances of the different algorithms  are reported.
    The dashed line reports the theoretical complexity of $\mathcal O(n\log^2(n))$.}
  \label{fig:laplacian}
\end{figure}

In Table \ref{tab:laplacian} and Figure~\ref{fig:laplacian}
we show the timings for the solution of this
problem for different grid sizes.
We stress that, since full matrices are never represented,
a large amount of RAM is not needed to run the solver. Nevertheless, this
is needed when using \texttt{lyap} from the Matlab Control Toolbox, so we
have comparisons with the latter only for $n \leq 4096$. 

The results in Table \ref{tab:laplacian} show that the timings
are just a little more than linear in the size of the problem.
Figure~\ref{fig:laplacian} illustrates that the complexity is in fact
$\mathcal O(n \log^2 n)$, for the methods that evaluate the integral formula \eqref{eq:integral}.

The approach based on the sign iteration is
faster than the one that exploits the integral formula.
Nevertheless the latter has a slightly better asymptotic cost since it requires $\mathcal O(n\log^2(n))$ flops 
instead of  $\mathcal O(n\log^3(n))$. Another advantage of the integral formula is the easy parallelization. 
In fact, the evaluation of the integrand at the nodes can be carried out in a parallel fashion on different machines or cores. 
In our tests we used $32$ integration nodes so the maximum gain in the performances can be obtained using $32$ cores. 
The results reported in Table~\ref{tab:laplacian} confirm the acceleration of the paralel implementation when using $12$ cores.

\subsection{The 2D heat equation}
\label{sec:heat-equation}

We consider now a case of more practical interest, which has been
described and studied by Haber and Verhaegen in
\cite{haber,haber2014estimation}.  They study a particular
discretization for the 2D heat equation that gives rise to a Lyapunov
equation with banded matrices. Let $S_m = \textrm{trid}(1,0,1)$ be
the $m \times m$ matrix with $1$ on the super and subdiagonal and zeros
elsewhere, and let $\mathbf 1_m\in\mathbb C^m$ be the vector with all
the entries equal to $1$. The resulting Lyapunov equation involves the coefficient matrices
\[
  A = I_m \otimes (a I_6 + e S_6) + e S_m \otimes I, \qquad 
  C = I_m \otimes (0.2 \cdot \mathbf 1_6 \mathbf 1_6^T + 0.8 I) + 0.1 S_m \otimes (\mathbf 1_6 \mathbf 1_6^T).
\]
For the details on how these matrices are obtained from the
discretization phase we refer to \cite{haber2014estimation}. The parameters
$a$ and $e$ are set to $a = 1.36$ and $e = -0.34$.
These two matrices
are banded, with bandwidth $6$ and $11$, respectively. However, a
careful look shows that the quasiseparable rank of $A$ is $6$, but the
one of $C$ is $1$: the quasiseparable representation can exploit more
structure than the banded one in this problem.

\begin{table}[t]
  \centering
  \small
  \begin{filecontents*}{experiment2-malc1.dat}
768	1.0623	8.9521e-13	12	1.9641	9.4411e-12	13	1464.8	1.1761	2.9557e-11	6015.8
1536	2.7421	1.4158e-12	12	4.9865	4.9164e-12	12	3169.8	2.4886	2.8103e-11	13122
3072	8.2859	9.7344e-12	10	13.116	1.527e-11	12	6915.6	4.7853	2.6656e-11	26637
6144	19.302	4.9406e-12	10	32.211	1.0755e-11	10	14023	9.2334	2.571e-11	52400
12288	48.44	4.7634e-12	10	79.455	1.3569e-11	10	29967	18.253	2.4122e-11	1.0151e+05
24576	117.32	4.7108e-12	10	189.84	1.8048e-11	10	63774	36.963	3.2172e-11	1.9046e+05
49152	277.8	1.0898e-11	11	445.03	1.6207e-11	10	1.3523e+05	67.178	3.0305e-11	3.6354e+05
98304	589.51	3.8725e-11	10	1092.1	2.6919e-11	10	2.8582e+05	121.31	2.8671e-11	6.9118e+05
1.9661e05	1312.6  1.0456e-10      10	2677.1	8.1602e-11	9	5.7163e+05	213.08	2.7499e-11	1.3095e+06	








\end{filecontents*}
\pgfplotstabletypeset[%
  sci zerofill,
  columns={0,1,2,4,5,6,8,9},
  columns/0/.style={column name=$n$},
    columns/1/.style={column name=$T_{\text{Sign}}$},
    columns/2/.style={column name=$Res_{\text{Sign}}$},
    columns/4/.style={column name=$T_{\text{ParExp}}$},
    columns/5/.style={column name=$Res_{\text{ParExp}}$},
    columns/6/.style={
    	column name=QS rk,
    	string replace={0}{}
    },
    columns/8/.style={column name=$T_{\text{SparseCG}}$},
    columns/9/.style={column name=$Res_{\text{SparseCG}}$}, 
      empty cells with={---}
  ]{experiment2-malc1.dat}
  \caption{Timings and features of the solution of the heat equation for
    different grid sizes. For the methods based on the HODLR arithmetic the minimum block size is set to $256$ and the relative threshold in truncation is $\epsilon=10^{-12}$.
    In this example the quasiseparable rank of the solution coincides for the implementation based on the sign function and on the integral formula,
    so we have only reported it once.
    }
  \label{tab:heat-equation}
\end{table}

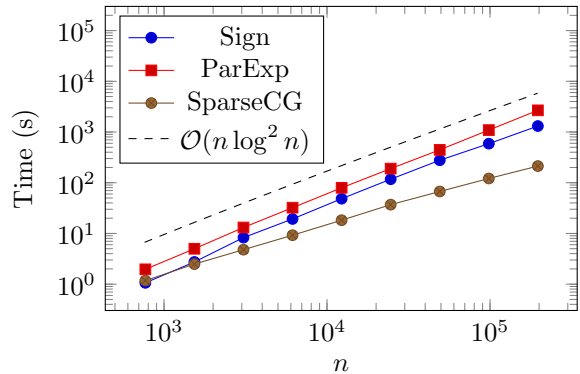
\begin{figure}
  \centering
  \begin{minipage}{.45\linewidth}
  	\pgfplotstabletypeset[%
  	sci zerofill,
  	columns={0,7,10},
  	columns/0/.style={column name=$n$},
  	columns/7/.style={column name=$Mem_{\text{HODLR}}$},
  	columns/10/.style={column name=$Mem_{\text{Sparse}}$},  
  	empty cells with={---}
  	]{experiment2-malc1.dat}
  \end{minipage}~\begin{minipage}{0.53\textwidth}  	
    \bigskip
    
  	\begin{tikzpicture}
    \begin{loglogaxis}[width=\linewidth, height=.27\textheight,
      legend pos = north west,
      xlabel = $n$, ylabel = Time (s), ymax = 3e5]
      \addplot table[x index=0, y index=1]  {experiment2-malc1.dat};
       \addplot table[x index=0, y index=4]  {experiment2-malc1.dat};
        \addplot table[x index=0, y index=8]  {experiment2-malc1.dat};  
        \addplot[domain=760:196000,dashed] {2e-4 * x * ln(x) * ln(x)};
      \legend{Sign, ParExp, SparseCG,  $\mathcal O(n \log^2 n)$};
    \end{loglogaxis}          
  \end{tikzpicture}
  \end{minipage}
  \caption{On the left, the memory consumption in storing the
  	solution of the heat equation computed with
  	\texttt{ParExp} and \texttt{SparseCG}, respectively. The
  	first exploit the HODLR representation while the second one
  	makes use of the sparse format. The numerical values
  	reported are in KB (Kilobytes). On the right, timings for the solution of the heat equation.} 
  \label{fig:heat-equation}
\end{figure}

We have solved this problem for different values of $m$, from $m = 128$
to $m = 32768$.  For each $m$, the size of the associated matrices $A$
and $P$ is $6m \times 6m$.
We have also compared our implementation to the (sparse)
conjugate gradient
implemented in matrix form, as proposed in \cite{Palitta2017}. One
can see that, at the $k$-th iteration of the CG method the solution
in matrix form has a bandwidth proportional to $k$; when the method
converges in a few steps, this can provide an accurate banded
approximation to the solution in linear time.
In fact, this problem is  well-conditioned
independently of $n$, and therefore is the ideal candidate for the
application of this method (as shown in \cite{Palitta2017}). Additionally, 
the sparse arithmetic implemented in MATLAB is very efficient, and
the computational cost is linear without any logarithmic factor.
Figure~\ref{fig:heat-equation} confirms the predicted
$\mathcal O(n \log^2 n)$ complexity for the methods that we propose.
The timings of the conjugate gradient are comparable to the
sign iteration for small dimensions, but then the
absence of the $\log^2(n)$ factor in the complexity is a
big advantage for the former method. 

All the proposed approaches seems to work better in terms of CPU time than
the one reported by Haber and Verhaegen in \cite{haber}, which use a
comparable (although slightly older) CPU. Moreover, their approach
is delivering only about two digits of accuracy with the selected parameter,
while we get solutions with a relative error of
about $10^{-10}$ in the Frobenius norm.

The table in Figure~\ref{fig:heat-equation} reports the memory
usage when the solution is stored in the HODLR and in the sparse formats.
We can see that the method using HODLR matrices, although
slower, is more memory efficient compared to the \texttt{SparseCG} of
a factor of about $2$. 

\subsection{Partial Integro-Differential equation}
Here, we consider a generalized Sylvester equation that has the
structure described by Corollary~\ref{cor:sylv-low-rank}, and arises from the discretization of the following partial differential equation
\begin{equation}\label{eq:diff-int}
-\Delta u (x,y) + q(x,y)\int_{[0,1]^2}r(x,y)u(x,y)\: dx\, dy=f(x,y) \qquad (x,y)\in (0,1)^2,
\end{equation}
where $q(x,y)=q_1(x)q_2(y)$ and $r(x,y)=r_1(x)r_2(y)$ are separable functions, and we assume zero Dirichlet boundary conditions.
The discrete operator can be expressed in terms of the matrix equation 
$AX+XA+M_1XM_2^T = C$ where $A= (n-1)^2\cdot \trid(-1,2,-1)$, 
$C$ is the sampling of $f$ over the uniform grid $x_j=y_j= \frac{j-1}{n-1}$, $j=1\dots,n$, and
\begin{align*}
M_1&= \frac{1}{n-1} \begin{bmatrix}
q_1(x_1)\\
q_1(x_2)\\
\vdots\\
q_1(x_{n-1})\\
q_1(x_n)
\end{bmatrix}\begin{bmatrix}
\frac 12r_2(x_1)\\
r_2(x_2)\\
\vdots\\
r_2(x_{n-1})\\
\frac 12r_2(x_n)
\end{bmatrix}^T,\qquad 
M_2= \frac{1}{n-1} \begin{bmatrix}
q_2(x_1)\\
q_2(x_2)\\
\vdots\\
q_2(x_{n-1})\\
q_2(x_n)
\end{bmatrix}\begin{bmatrix}
\frac 12r_1(x_1)\\
r_1(x_2)\\
\vdots\\
r_1(x_{n-1})\\
\frac 12r_1(x_n)
\end{bmatrix}^T.
\end{align*}
In this experiment we consider $f(x,y)=\log(1 +|x-y|)$,  and $q_j(x)=r_j(x)\equiv \sin(3x)$ $j=1,2$ so that $M_1=M_2$.
We  test Algorithm~\ref{alg:swm} and the results are reported in 
Figure~\ref{fig:part-integro}. The results  report the timings of Algorithm~\ref{alg:swm} using the sign
iteration for solving the first quasiseparable Lyapunov equation and one can see that
the timings of the overall procedure are just slightly larger than the ones reported for Example \ref{Laplace_ex}.
Indeed, step \ref{first_sol} of Algorithm~\ref{alg:swm} is the dominating cost of the whole computation. 

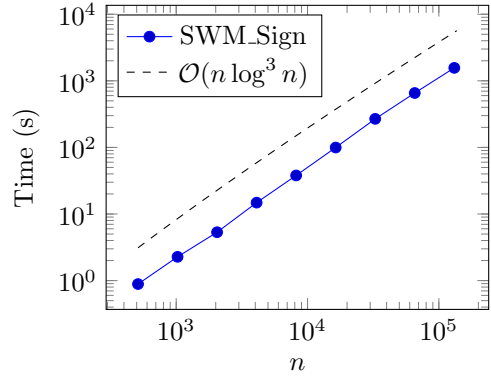
\begin{figure}
  \centering
  \begin{minipage}{0.5\textwidth}
     \begin{filecontents*}{experiment4-malc1.dat}
512     0.88707 4.2049e-12      12	1.71	1.5036e-07	12
1024    2.2745  5.1823e-12      10	3.4078	1.6815e-08	13
2048    5.3289  1.1057e-11      10	7.3349	3.465e-09	13
4096    14.816  4.6233e-11      11	16.541	2.029e-10	12
8192    37.929  3.0903e-11      11	36.863	2.4597e-11	10
16384   99.631  6.4915e-11      12	82.014	6.5006e-12	10
32768   268.37  1.188e-10       14	182.52	2.3501e-12	12
65536   656.41  1.4308e-10      14	406.7	1.0277e-12	10
1.3107e+05      1565    1.5226e-10      14	904.19	6.7119e-13	9


\end{filecontents*}
\pgfplotstabletypeset[%
     sci zerofill,
     columns={0,1,2,3},
     columns/0/.style={column name=$n$},
     columns/1/.style={column name=$T_{\text{Sign}}$},
     columns/2/.style={column name=$Res_{\text{Sign}}$},
     columns/3/.style={column name=QS rk},
     columns/4/.style={column name=$T_{\text{ParPad\'e}}$},
     columns/5/.style={column name=$Res_{\text{ParPad\'e}}$},
     columns/6/.style={column name=QS rank}
     ]{experiment4-malc1.dat}
  \end{minipage}~\begin{minipage}{0.45\textwidth}
  \bigskip 
  \begin{tikzpicture}
  \begin{loglogaxis}[width=\linewidth, height=.27\textheight,
  legend pos = north west,
  xlabel = $n$, ylabel = Time (s)]
  \addplot table[x index=0, y index=1]  {experiment4-malc1.dat};
  \addplot[domain=512:137000,dashed] {2.5e-5 * x * ln(x) * ln(x) * ln(x)};
  \legend{SWM\_Sign,$\mathcal O(n \log^3 n)$};
  \end{loglogaxis}
  \end{tikzpicture}
\end{minipage}
   \caption{On the left, timings and features of the solution of the partial integro-differential equation for different grid sizes. The minimum block size is set to $256$ and relative threshold in truncation is $\epsilon=10^{-12}$.   	
   	On the right, we plot timings for the solution of the generalized Lyapunov equation coming from the partial-integro differential equation.
    }
   \label{fig:part-integro}
  \end{figure}

\section{Final remarks}\label{sec:conclusion}
We have compared and analyzed two different strategies for the
solution of some linear matrix equations with rank structured
data. We have presented some theoretical results that justify
the feasibility of the approaches relying on tools from rational
approximation. The techniques that we developed can be applied to
treat the case of banded matrix coefficients in a natural way, thus
providing an alternative approach to the ones presented in
\cite{haber,Palitta2017}. Moreover, our methods still perform well when the
conditioning of the coefficients increases. This allows to cover a wider
set of problems related to PDEs, such as the ones including the Laplacian operator. 

Numerical tests confirm the scalability of the approach in treating
large scale instances. Our experiments show that the sign iteration is
usually the fastest and most accurate method, although the procedure
based on the integral formula can be more effective in parallel
environments.

In the case of the asymptotically ill-conditioned coefficients
in the matrix equation (such as for the 2D Laplacian), the complexity
of the sign iteration is slightly worse than the one of the
integration formula ($\mathcal O(n \log^3 n)$ instead
of $\mathcal O(n \log^2 n)$). This can make the latter method
the most attractive choice for really large scale problems. Moreover, relying on HSS arithmetic \cite{chandra}, in place of HODLR, 
would likely further improve the proposed
approach. This will be subject to future 
investigation. The main difficulty lies in 
creating a fast and reliable procedure for
the computation of the inverse in HSS format.

\section{Acknowledgments}
We thank Daniel Kressner for useful discussions and suggestions. 
 All the authors are members of the INdAM Research group GNCS and
this work has been supported by the GNCS/INdAM project ``Metodi numerici avanzati per equazioni e funzioni di matrici
con struttura'', by the Region of Tuscany (PAR-FAS 2007 -- 2013) and by MIUR,
the Italian Ministry of Education, Universities and Research (FAR) within the Call FAR -- FAS 2014
(MOSCARDO Project: ICT technologies for structural monitoring of age-old constructions based
on wireless sensor networks and drones, 2016 -- 2018). Their support is gratefully
acknowledged.

%
\bibliographystyle{abbrv}
\bibliography{library}

\end{document}